\documentclass{article}
\usepackage[english]{babel}
\usepackage[utf8]{inputenc}
\usepackage{johd}

\usepackage{lineno,hyperref}
\usepackage[left=3.5cm,top=2.5cm,right=3.5cm,bottom=2.5cm]{geometry}
\usepackage{graphicx,array,delarray}
\usepackage{subfigure}
\usepackage{latexsym}
\usepackage{amscd}
\usepackage{amsfonts, amsmath, amssymb}
\usepackage[utf8]{inputenc}
\usepackage{verbatim}
\usepackage{mathrsfs}
\usepackage{fancyhdr}
\usepackage{float}
\usepackage{palatino}
\usepackage{mathtools}
\usepackage{amsmath}
\usepackage{amssymb}

\modulolinenumbers[5]

\newtheorem{theorem}{\textbf{Theorem}}

\newtheorem{definition}[theorem]{\textbf{Definition}}

\newtheorem{proposition}[theorem]{\textbf{Proposition}}
\newtheorem{remark}[theorem]{Remark}

\numberwithin{equation}{section}

\newenvironment{proof}[1][Proof]{\noindent\textbf{#1.} }{\ \rule{0.5em}{0.5em}}

\title{Generalized Reproducing Kernel Banach Spaces: A Functional Analytic Framework for Abstract Neural Networks}

\author{Raul Felipe-Sosa$^{a}$$^{,*}$ \\
        \small $^{a}$Universidad Antonio de Nebrija, Madrid, Spain \\
        \small $^{*}$Corresponding author: R. Felipe-Sosa; \tt{rfelipe@nebrija.es} \\
}
\date{}

\begin{document}
\maketitle

\begin{abstract}
\noindent
In this paper, we introduce a generalization of Reproducing Kernel Banach Spaces (RKBS), 
which we term \emph{Generalized Reproducing Kernel Banach Spaces} (GRKBS). 
The motivation stems from recent results showing that classical fully connected neural networks 
can be understood as finite-dimensional subspaces of RKBS. 
Our generalization extends this perspective to settings with Banach-valued codomains, 
allowing the construction of \emph{abstract neural networks} (AbsNN) as compositions of GRKBS. 
This framework provides a natural pathway to model neural architectures that go beyond classical machine learning paradigms, including physically-informed structures governed by differential equations. 
We establish a unified definition of GRKBS, prove structural uniqueness results, 
and analyze the existence of sparse minimizers for the corresponding abstract training problem. 
This contributes to bridging functional analytic theory and the design of new neural architectures with applications in 
both approximation theory and mathematical modeling.
\end{abstract}

\noindent

\section{Introduction}

The interplay between functional analysis and machine learning has gained increasing attention in recent years. 
In particular, Reproducing Kernel Hilbert Spaces (RKHS) have long provided a solid theoretical foundation 
for kernel methods and statistical learning theory, see \cite{Hofmann2008, Kanagawa2025}. More recently, the concept of Reproducing Kernel Banach Spaces (RKBS) 
has been developed \cite{Xu2019, Bertolucci2023}, extending the RKHS framework to Banach space settings. 
A key insight from this line of research is that fully connected neural networks can be interpreted as finite-dimensional 
Banach subspaces of suitable RKBS, thereby endowing neural architectures with a rigorous functional-analytic interpretation.

Despite these advances, the classical theory of RKBS remains limited in scope. 
In particular, most formulations restrict to scalar-valued function spaces, 
which constrains their applicability to modern architectures that operate in structured, often infinite-dimensional, codomains. 
Moreover, existing definitions of RKBS coexist without a fully unified framework, which complicates both theoretical analysis 
and practical implementation.

The aim of this paper is to address these limitations by introducing the notion of 
\emph{Generalized Reproducing Kernel Banach Spaces} (GRKBS). 
Within this setting, we define \emph{abstract neural networks} (AbsNN) as suitable compositions of GRKBS, 
generalizing the construction of deep architectures beyond classical feedforward models. 
This approach makes it possible to embed physical models—such as those governed by differential equations—into the neural structure, 
yielding architectures that are qualitatively different from standard ones.

Our contributions can be summarized as follows:
\begin{itemize}
\item We introduce a generalized and unified definition of RKBS, extending it to Banach-valued codomains.
\item We establish what we regard as the configuration of a GRKBS, showing that a GRKBS may admit infinitely many configurations grouped into equivalence classes, based on an equivalence relation defined on the set of configurations. 
\item In the generalized framework we propose, we unify the existing definitions of RKBS. In this sense, we prove that a reproducing kernel for a GRKBS corresponds to a particular configuration of the space.
\item We present uniqueness results for the reproducing kernel of a GRKBS.
\item We define abstract neural networks as compositions of GRKBS and provide a functional-analytic description of multi-layer structures.
\item We analyze the abstract training problem (ATP) in this setting and prove the existence of sparse minimizers, following ideas from variational analysis \cite{Bredies2020}.
\end{itemize}

This work contributes to the ongoing effort of building a rigorous mathematical foundation for neural networks 
and opens the door to new architectures where learning is intertwined with physical and functional-analytic models.

The structure of the paper is as follows. 
In Section~\ref{secc1} we present the main definitions and preliminary results, 
highlighting the connection between classical RKBS and neural networks. 
Section~\ref{sec:GRKBS} introduces the concept of Generalized Reproducing Kernel Banach Spaces (GRKBS) 
and establishes their basic properties, including uniqueness results and the construction of reproducing kernels. 
In Section~\ref{sec:AbsNN}, we define abstract neural networks (AbsNN) as compositions of GRKBS 
and analyze how this framework naturally leads to multi-layer architectures. 
Finally, in Section~\ref{sec:ATP} is devoted to the study of the abstract training problem in GRKBS and AbsNN, 
where we prove the existence of sparse minimizers by leveraging tools from variational analysis. 

\section{Some Definitions and Preliminary Results}\label{secc1}
This section aims to present essential definitions and results that are crucial for the development of this paper. In particular, we focus on presenting a theoretical framework in which the training process of a deep abstract neural structure can be explicitly formulated using given training data. Additionally, we briefly discuss the connection between reproducing kernel Banach spaces and fully connected neural networks, a relationship that has been previously explored in the literature; see \cite{Bertolucci2023} and \cite{Long2022}.

In the context of this work, $\mathcal{X}$ will be treated as an arbitrary set, with the subsets $\left\{x_i\right\}^N_{i = 1} \subset \mathcal{X}$ and $\left\{y_i\right\}^N_{i = 1} \subset \mathbb{R}$ referred to as training data. It is important to note that $\mathcal{X}$ is not assumed to have a geometric or topological structure. For instance, it could be a topological space, a metric space, or similar, which would provide a different approach to studying the spaces we will later define, although this is not the focus of our investigation.

Additionally, we will consider a function $\mathcal{L}: \mathbb{R}^2 \rightarrow \mathbb{R}$, which adheres to specific properties that will be defined subsequently. Within the realm of machine learning, $\mathcal{L}$ is referred to as a cost function.

Let's give the following definition.

\begin{definition}\label{def1}
Let $\mathcal{B}$ be a Banach space consisting of functions $f: \mathcal{X} \rightarrow \mathbb{R}$. We refer to the \textbf{abstract training problem} (ATP) associated with $\mathcal{B}$ and the training dataset $\left\{x_i\right\}_{i = 1}^{N} \subset \mathcal{X}$ and $\left\{y_i\right\}_{i = 1}^{N} \subset \mathbb{R}$ as:
\begin{align}\label{eqn1}
\inf_{f \in \mathcal{B}}\left(\frac{1}{N} \sum_{i = 1}^{N} \mathcal{L}(y_i, f(x_i)) + \|f\|_{\mathcal{B}}\right).
\end{align}
\end{definition}

Definition \ref{def1} provides a formal structure for the concept of model training in machine learning. In particular, if we take $\mathcal{X} = \mathbb{R}^n$ and consider the class of functions $f_{\mathcal{N}}: \mathbb{R}^{n} \rightarrow \mathbb{R}$ given by
\begin{align*}
f_{\mathcal{N}}(x) = \sum_{k = 1}^{K} \alpha_k \sigma(w_k \cdot x + b_k),
\end{align*}
where $\alpha_k, b_k \in \mathbb{R}$, $w_k \in \mathbb{R}^n$, and $\cdot$ denotes the standard inner product in $\mathbb{R}^n$, and endow this function class with a norm that turns it into a Banach space, then solving the abstract training problem corresponds to training a neural network with a single hidden layer of up to $N$ neurons and an input layer of $n$ neurons. Thus, a neural network model can be viewed as a Banach space in which an abstract training problem of the form \eqref{eqn1} is posed and solved.

The final point discussed in the preceding paragraph arises due to the connection between what are known as reproducing kernel Banach spaces (RKBS) and neural networks. RKBS extend the concept of reproducing kernel spaces. In general, these spaces are Hilbert spaces characterized by the presence of a reproducing kernel that can, in some manner, reproduce every function within the space. The following definition introduces what is known as a reproducing kernel Banach space (refer to Definition 3.2 of \cite{Bertolucci2023}).

\begin{definition}\label{def2}
Consider $\mathcal{B}$ as a Banach space of functions $f: \mathcal{X} \rightarrow \mathbb{R}$, where the operations of addition and scalar multiplication are conducted pointwise. The space $\mathcal{B}$ is classified as a \textbf{\textit{reproducing kernel Banach space}} (RKBS) if, for each $x \in \mathcal{X}$, the evaluation linear functional that assigns to every $f \in \mathcal{B}$ the specific value $f(x)$, denoted by $ev_x$, is both linear and continuous, that is, by $ev_{x} \in \mathcal{B}'$.
\end{definition}
Note that Definition \ref{def2} does not explicitly reference any component of a reproducing kernel. However, as shown in the following proposition, the reproducing property is inferable from the so-called characteristic space and characteristic map. Proposition 3.3 from \cite{Bertolucci2023} offers another definition of RKBS, effectively serving as a guideline to build such spaces starting from the characteristic space and characteristic map. Let us examine this result.

\begin{proposition}\label{prop1}
Let $\mathcal{B}$ be a Banach space of functions as in Definition \ref{def2}. $\mathcal{B}$ is an RKBS if and only if there exist a Banach space $\mathcal{F}$ and a map $\phi: \mathcal{X} \rightarrow \mathcal{F}'$ such that the following conditions are satisfied:
\begin{align*}
&\mathcal{B} = \left\{f_{\mu}: \mu \in \mathcal{F}\right\}, \;\; \text{and}\;\; f_{\mu}(x) = \prescript{}{\mathcal{F}'}{\left\langle \phi(x), \mu\right\rangle}_{\mathcal{F}},\\
&\|f\|_{\mathcal{B}} = \inf\left\{\|\mu\|_{\mathcal{F}}: f = f_{\mu}\right\}.
\end{align*}
\end{proposition}
Note that Proposition \ref{prop1} indicates that, with a given $\mathcal{F}$ and $\phi$, the corresponding RKBS is uniquely specified. Thus, this proposition offers a methodology for constructing an RKBS.

In what follows, we present an example of how to construct an RKBS from a given Banach space $\mathcal{F}$ and a map $\phi$. This example can be found in \cite{Bertolucci2023}.

Let us consider $\mathcal{X} = \mathbb{R}^n$ with $\Theta = \mathbb{R}^{n + 1}$. Define $\mathcal{M}(\Theta)$ as the collection of bounded measures $\mu$ over the Borel $\sigma$-algebra of $\Theta$. Using the total variation norm $\|\cdot\|_{TV}$, $\mathcal{M}(\Theta)$ serves as a Banach space, and we set $\mathcal{F} = \mathcal{M}(\Theta)$. To define the function $\phi$, we introduce two functions $\rho: \mathcal{X}\times \Theta \rightarrow \mathbb{R}$ and $\beta: \Theta \rightarrow \mathbb{R}$, which adhere to these conditions:,\begin{itemize}
    \item For all $x \in \mathcal{X}$, it holds that $\rho(x, \cdot)\beta(\cdot) \in C_0(\Theta),$ where $C_0(\Theta)$ consists of continuous functions on $\Theta$ that tend to zero at infinity,
    \item For every $x \in \mathcal{X}, \rho(x, \cdot)$ must be measurable.
\end{itemize}
Next, we define $\phi: \mathcal{X} \rightarrow \mathcal{M}(\Theta)'$ by specifying that
\begin{align}\label{phiBN}
\prescript{}{\mathcal{M}(\Theta)'}{\left\langle \phi(x), \mu\right\rangle}_{\mathcal{M}(\Theta)} := \int_{\Theta}\rho(x, \theta)\beta(\theta)d\mu(\theta),\;\; \forall \mu \in \mathcal{M}(\Theta), x \in \mathcal{X}.
\end{align}
Thus, the function $\phi$ is established.

The conditions imposed on the functions $\rho$ and $\beta$ ensure that the integral in the definition of $\phi$ exists. 

Let $\mathcal{B}_{\mathcal{N}\mathcal{N}}$ denote the RKBS associated with $\mathcal{F} = \mathcal{M}(\Theta)$ and the map $\phi$ constructed above. Then,
\begin{align*}
\mathcal{B}_{\mathcal{N}\mathcal{N}} &= \left\{ f_{\mu} : \mu \in \mathcal{M}(\Theta) \right\}, \\
f_{\mu}(x) &= \int_{\Theta} \rho(x, \theta)\beta(\theta)\, d\mu(\theta), \quad \text{for all } x \in \mathcal{X}, \\
\|f\|_{\mathcal{B}_{\mathcal{N}\mathcal{N}}} &= \inf \left\{ \|\mu\|_{TV} : f = f_{\mu} \right\}.
\end{align*}
This RKBS is closely related to fully connected feedforward neural networks, specifically feedforward MLPs with a single hidden layer. According to Theorem 3.9 in \cite{Bertolucci2023}, if we let $\mathcal{N}_{\text{net}}$ represent the group of functions of an MLP with one hidden layer, then $\mathcal{N}_{\text{net}}$ forms a finite-dimensional Banach subspace of $\mathcal{B}_{\mathcal{N}\mathcal{N}}$, with its dimension being $K \leq N$. Furthermore, solutions to the abstract training problem linked to $\mathcal{B}_{\mathcal{N}\mathcal{N}}$ are located within $\mathcal{N}_{\text{net}}$. Simply put, solving the abstract training problem related to $\mathcal{B}_{\mathcal{N}\mathcal{N}}$ is equivalent to training an MLP. 

This is a key result that motivates everything that follows, as it provides a methodology for constructing discrete neural structures as Banach subspaces of RKBS. On the one hand, this endows neural networks with a theoretical interpretation they previously lacked: namely, they can be viewed as finite-dimensional Banach spaces, where the dimension corresponds to the number of neurons in the network's hidden layer. In this sense, it becomes natural to consider generalizing the concept of RKBS in order to construct neural structures that are qualitatively different from classical ones. 

\section{Generalized Reproducing Kernel Banach Spaces}\label{sec:GRKBS}
In this section, we develop an extension of the concept of Reproducing Kernel Banach Spaces (RKBS). Our aim is to include spaces of functions with more general codomains, encompassing Banach spaces beyond $\mathbb{R}$, including infinite-dimensional settings. From the standpoint of the definition itself—by analogy with Definition~\ref{def2}—this simply amounts to requiring that pointwise evaluation operators be bounded. However, when considering the underlying space, the associated feature map, and the resulting theoretical framework, a more in-depth analysis is required.

In contrast to classical RKBS, these Generalized Reproducing Kernel Banach Spaces (GRKBS) will exhibit a broader connection with neural network architectures, allowing for the inclusion of multiple hidden layers of arbitrary finite dimension, as well as hidden layers of infinite dimension. In this manuscript, we present an example of a GRKBS associated with a neural network featuring a single hidden layer, whose output is given by the solution of a partial differential equation simulating the electrical activation in a specific region of the brain. Our aim is to introduce neural networks that are qualitatively different from conventional models, not only to perform prediction or classification tasks, but also to emulate the physical and physiological processes by which the brain detects, encodes, and responds to surrounding environmental information. We refer to these architectures as \emph{physically-modeled neural networks}. For this type of networks, it is reasonable to consider function spaces with a broader codomain, such as a Banach space. 

Throughout this paper, let $\mathit{E}$ be a Banach space, and let $\mathcal{B}$ be a Banach space of functions $f: \mathcal{X} \rightarrow \mathit{E}$, endowed with pointwise addition and scalar multiplication. For each $x \in \mathcal{X}$, we define the evaluation operator $ev_x: \mathcal{B} \rightarrow \mathit{E}$ as the linear map that sends each $f \in \mathcal{B}$ to its value at $x$, that is, $ev_x(f) = f(x)$.

Moreover, we denote by $L_{\mathrm{op}}(\mathcal{B}, \mathit{E})$ the space of bounded linear operators from $\mathcal{B}$ into $\mathit{E}$.

Finally, given two Banach spaces, denoted by $\mathcal{B}$ and $\mathcal{F}$, we say that they are \emph{isometrically isomorphic}, denoted by $\mathcal{B} \cong \mathcal{F}$, if there exists an isometric isomorphism from $\mathcal{B}$ to $\mathcal{F}$.

\begin{definition}\label{def3}
We say that $\mathcal{B}$ is a \textbf{generalized reproducing kernel Banach space (GRKBS)} if, for every $x \in \mathcal{X}$, the evaluation operator $ev_x$ is bounded. In this case, we write $ev_x \in L_{\mathrm{op}}(\mathcal{B}, \mathit{E})$.
\end{definition}
Following Proposition 3.3 from \cite{Bertolucci2023}, we will prove the following proposition.

\begin{proposition}\label{prop2} 
Let $\mathit{E}$ be a Banach space and let $\mathcal{B}$ be a vector space of functions as in Definition \ref{def3}. $\mathcal{B}$ is a GRKBS if and only if there exist a Banach space $\mathcal{F}$ and a mapping $\phi: \mathcal{X} \rightarrow L_{\mathrm{op}}(\mathcal{F}, \mathit{E})$, such that:
\begin{align}
&\mathcal{B} = \left\{f_{\mu}: \mu \in \mathcal{F}\right\},\; \text{where}\; f_{\mu}(x) = \phi(x)(\mu), \forall x \in \mathcal{X}, \mu \in \mathcal{F}\label{eqn2},\\
&\|f\|_{\mathcal{B}} = \inf\left\{\|\mu\|_{\mathcal{F}}: f = f_{\mu}\right\}\label{eqn3}.
\end{align}
\end{proposition}
\begin{proof}
First, suppose that $\mathcal{B}$ is a GRKBS as defined in Definition \ref{def3}. Let $\mathcal{F} = \mathcal{B}$ and define $\phi: \mathcal{X} \rightarrow L_{\mathrm{op}}(\mathcal{B}, \mathit{E})$ by $\phi(x) = ev_x$. Clearly, this choice of $\mathcal{F}$ and $\phi$ satisfies conditions \eqref{eqn2} and \eqref{eqn3}.

Conversely, suppose that $\mathcal{B}$ is such that there exist $\mathcal{F}$ and $\phi$ satisfying \eqref{eqn2} and \eqref{eqn3}. We aim to show that $\mathcal{B}$ is a GRKBS in the sense of Definition \ref{def3}.

To this end, define the linear operator $\phi_{*} : \mathcal{F} \rightarrow \mathcal{B}$ by $\phi_{*}(\mu) = f_{\mu}$. This operator is surjective, since condition \eqref{eqn2} is assumed to hold. However, it is not necessarily injective, as a given $f \in \mathcal{B}$ may admit multiple representations $f = f_{\mu}$ with different $\mu \in \mathcal{F}$. 

Now define
\begin{align}\label{defN}
\mathcal{N}_{\phi} = \ker \phi_{*} = \bigcap_{x \in \mathcal{X}} \ker \phi(x),
\end{align}
which is a closed subspace of $\mathcal{F}$. The quotient space $\mathcal{F}/\mathcal{N}_{\phi}$ is a Banach space equipped with the norm
\[
\|\mu + \mathcal{N}_{\phi}\|_{\mathcal{F}/\mathcal{N}_{\phi}} = \inf \left\{ \|\nu\|_{\mathcal{F}} : \nu \in \mu + \mathcal{N}_{\phi} \right\}, \quad \text{for all } \mu + \mathcal{N}_{\phi} \in \mathcal{F}/\mathcal{N}_{\phi}.
\]
Define again $\phi_{*} : \mathcal{F}/\mathcal{N}_{\phi} \rightarrow \mathcal{B}$ by $\phi_{*}\left(\mu + \mathcal{N}_{\phi}\right) = f_{\mu}$. Then $\phi_{*}$ is an isomorphism from $\mathcal{F}/\mathcal{N}_{\phi}$ onto $\mathcal{B}$, and by construction of the norm on the quotient, this isomorphism is an isometry. Since $\mathcal{F}/\mathcal{N}_{\phi}$ is a Banach space, it follows that $\mathcal{B}$ is also a Banach space.

Finally, for every $x \in \mathcal{X}$ and $f \in \mathcal{B}$, we have
\[
\|ev_x(f)\|_{\mathit{E}} = \|f_{\mu}(x)\|_{\mathit{E}} = \|\phi(x)(\mu)\|_{\mathit{E}}, \quad \text{for all } \mu \in \mathcal{F} \text{ such that } f = f_{\mu}.
\]
Hence,
\[
\|ev_x(f)\|_{\mathit{E}} \leq \|\phi(x)\|_{L(\mathcal{F}, \mathit{E})} \|\mu\|_{\mathcal{F}} \leq \|\phi(x)\|_{L(\mathcal{F}, \mathit{E})} \inf_{\mu \in \mathcal{F} : f = f_{\mu}} \|\mu\|_{\mathcal{F}} = \|\phi(x)\|_{L(\mathcal{F}, \mathit{E})} \|f\|_{\mathcal{B}},
\]
which shows that $ev_x \in L_{\mathrm{op}}(\mathcal{B}, \mathit{E})$.
\end{proof}

In this proof, the space $\mathcal{N}_{\phi}$ is referred to as the kernel of $\phi$.
\begin{remark}\label{remark3}
Proposition~\ref{prop2} not only provides an alternative definition of a GRKBS, but also shows that a GRKBS is uniquely determined by the pair $(\mathcal{F}, \phi)$. In this sense, the proposition provides a direct method for constructing GRKBSs by relating them to a different Banach space, $\mathcal{F}$, which is not a function space and which, through the characteristic map $\phi$, parametrizes the space $\mathcal{B}$.

We refer to the pair $(\mathcal{F}, \phi)$ as a \textbf{\emph{configuration}} of the GRKBS, where $\mathcal{F}$ is called the \emph{configuration space}, and $\phi$ is the \textbf{\emph{configuration map}}.
\end{remark}
On the other hand, define
\begin{align*}
\mathcal{N}_{\phi}^{\bot} = \left\{G \in L(\mathcal{F}, \mathit{E}) : G(\mu) = 0 \in \mathit{E},\ \forall \mu \in \mathcal{N}_{\phi} \right\}.
\end{align*}
In connection with this set, we present the following proposition.

\begin{proposition}\label{prop4}
The isometric isomorphism $\phi_{*} : \mathcal{F}/\mathcal{N}_{\phi} \rightarrow \mathcal{B}$ established in the proof of Proposition~\ref{prop2} induces an isometric isomorphism from $\mathcal{N}_{\phi}^{\bot}$ to $L_{\mathrm{op}}(\mathcal{B}, \mathit{E})$.
\end{proposition}

\begin{proof}
Let us define the linear operator $\phi'_{*} : \mathcal{N}_{\phi}^{\bot} \rightarrow L(\mathcal{F}/\mathcal{N}_{\phi}, \mathit{E})$, which assigns to each $\mathcal{A} \in \mathcal{N}_{\phi}^{\bot}$ the operator $\mathcal{A}' \in L(\mathcal{F}/\mathcal{N}_{\phi}, \mathit{E})$ defined by $\mathcal{A}'(\mu + \mathcal{N}_{\phi}) = \mathcal{A}(\mu)$. It is straightforward to verify that $\phi'_{*}$ is linear. Moreover, since $\mathcal{A} \in \mathcal{N}_{\phi}^{\bot}$, the operator $\mathcal{A}'$ is well defined. Indeed, if $\nu \in \mu + \mathcal{N}_{\phi}$, then $\mathcal{A}(\nu) = \mathcal{A}(\mu)$.

Now, let us prove that $\phi'_{*}$ is an isomorphism. Suppose that $\mathcal{A} \in \mathcal{N}_{\phi}^{\bot}$ satisfies $\phi'_{*}(\mathcal{A}) \equiv \mathbf{0}$ in $L(\mathcal{F}/\mathcal{N}_{\phi}, \mathit{E})$. This means that the associated operator $\mathcal{A}'$ satisfies $\mathcal{A}'(\mu + \mathcal{N}_{\phi}) = \mathcal{A}(\mu) = 0$ for all $\mu \in \mathcal{F}$. Hence, $\mathcal{A} \equiv \mathbf{0}$, and we conclude that $\phi'_{*}$ is injective. 

On the other hand, given an operator $\mathcal{A}' \in L(\mathcal{F}/\mathcal{N}_{\phi}, \mathit{E})$, we can define an operator $\mathcal{A} \in \mathcal{N}_{\phi}^{\bot}$ that is constant on each equivalence class of the quotient $\mathcal{F}/\mathcal{N}_{\phi}$, i.e., $\left.\mathcal{A}\right|_{\mu + \mathcal{N}_{\phi}} = \mathcal{A}'(\mu + \mathcal{N}_{\phi})$. Consequently, we have $\left.\mathcal{A}\right|_{\mathcal{N}_{\phi}} = \mathcal{A}'(\mathcal{N}_{\phi}) = 0$, so that $\mathcal{A} \in \mathcal{N}_{\phi}^{\bot}$. This follows from the linearity of $\mathcal{A}'$, which implies that $\mathcal{A}'(\mathcal{N}_{\phi}) = 0$. From all of the above, we conclude that $\mathcal{A} = \left(\phi'_{*}\right)^{-1}(\mathcal{A}')$, proving that $\phi'_{*}$ is surjective. 

To show that $\phi'_{*}$ is an isometry, we verify that
\begin{align*}
\|\mathcal{A}'\|_{L(\mathcal{F}/\mathcal{N}_{\phi}, \mathit{E})} &= \sup_{\mu + \mathcal{N}_{\phi} \in \mathcal{F}/\mathcal{N}_{\phi}} \frac{\|\mathcal{A}'(\mu + \mathcal{N}_{\phi})\|_{\mathit{E}}}{\|\mu + \mathcal{N}_{\phi}\|_{\mathcal{F}/\mathcal{N}_{\phi}}} = \sup_{\mu + \mathcal{N}_{\phi} \in \mathcal{F}/\mathcal{N}_{\phi}} \frac{\|\mathcal{A}(\mu)\|_{\mathit{E}}}{\inf_{\nu \in \mu + \mathcal{N}_{\phi}} \|\nu\|_{\mathcal{F}}} \\
&= \sup_{\mu + \mathcal{N}_{\phi} \in \mathcal{F}/\mathcal{N}_{\phi}} \left( \sup_{\nu \in \mu + \mathcal{N}_{\phi}} \frac{\|\mathcal{A}(\nu)\|_{\mathit{E}}}{\|\nu\|_{\mathcal{F}}} \right) = \sup_{\nu \in \mathcal{F}} \frac{\|\mathcal{A}(\nu)\|_{\mathit{E}}}{\|\nu\|_{\mathcal{F}}} = \|\mathcal{A}\|_{L(\mathcal{F}, \mathit{E})}.
\end{align*}
Therefore, we have shown that $\phi'_{*}$ is an isometry. 

Moreover, it follows that the isometry $\phi_{*} : \mathcal{F}/\mathcal{N}_{\phi} \rightarrow \mathcal{B}$ induces a natural isometry, which we denote by $\widetilde{\phi}'_{*}$, from $L(\mathcal{F}/\mathcal{N}_{\phi}, \mathit{E})$ to $L_{\mathrm{op}}(\mathcal{B}, \mathit{E})$. Indeed, given $\mathcal{A} \in L(\mathcal{F}/\mathcal{N}_{\phi}, \mathit{E})$, we define $\mathcal{A}^{\bot} \in L_{\mathrm{op}}(\mathcal{B}, \mathit{E})$ by
\[
\mathcal{A}^{\bot}(f) = \mathcal{A}(\phi_{*}^{-1}(f)),
\]
where $\phi_{*}^{-1}(f) = \mu + \mathcal{N}_{\phi}$ for any $\mu \in \mathcal{F}$ such that $f = f_{\mu}$. It is straightforward to verify that the operator $\widetilde{\phi}'_{*}$ is an isometry.

Finally, we obtain the required isometry from $\mathcal{N}_{\phi}^{\bot}$ to $L_{\mathrm{op}}(\mathcal{B}, \mathit{E})$ as the composition $\phi_{*}^{\bot} = \widetilde{\phi}'_{*} \circ \phi'_{*}$.
\end{proof}

The analysis presented above leads to the following proposition, which will also play a role later in the discussion of the uniqueness of the structure defining a GRKBS.

\begin{proposition}\label{prop3}
Under the conditions of Proposition \ref{prop2}, it follows that $\mathcal{B}$ is a GRKBS if and only if there exist $(\mathcal{F}, \phi)$ as in Proposition \ref{prop2} such that
\begin{align}
\mathcal{B} \cong \mathcal{F}/\mathcal{N}_{\phi},
\end{align}
which is equivalent to stating that
\begin{align}
\mathcal{N}_{\phi}^{\perp} \cong L_{\mathrm{op}}(\mathcal{B}, \mathit{E}).
\end{align}
Here, the symbol $\cong$ denotes an isometric isomorphism between the corresponding Banach spaces.
\end{proposition}
This proposition provides two verifiable conditions that determine whether a given pair $(\mathcal{F}, \phi)$ constitutes a configuration of a specific GRKBS $\mathcal{B}$.

We are therefore interested in presenting several examples that are particularly relevant to applications, due to their connection with fully connected neural networks. Before doing so, we first aim to unify the definitions of RKBS that have appeared in the literature. In particular, we identify two main formulations: the one presented in \cite{Bertolucci2023}, which constitutes the primary basis of our framework; and the one given in \cite{Xu2019}, which explicitly centers on the reproducing kernel. 

Following Definition 2.1 in \cite{Xu2019}, we introduce the following definition, which serves as a generalization of that formulation.
\begin{definition}\label{def6}
Let $\Omega$ be a set and $\Omega'$ a Hausdorff topological space. Let $\mathcal{B}$ be a Banach space of functions $f: \Omega \rightarrow \mathit{E}$. Additionally, suppose there exists a function $K: \Omega \times \Omega' \rightarrow \mathit{E}$ such that, for every $x \in \Omega$, the function $K(x, \cdot)$ is measurable with respect to the variable in $\Omega'$. We say that $\mathcal{B}$ is a \textit{left-sided generalized reproducing kernel Banach space} (left-sided GRKBS), and that $K$ is its \textit{left-sided reproducing kernel}, if the following conditions are satisfied:
\begin{description}
    \item[i)] There exists a normed vector space $\mathcal{G}$ of functions $g: \Omega' \rightarrow \mathit{E},$ such that $\forall x \in \Omega$ $K(x, \cdot) \in \mathcal{G},$ and an isometric isomorphism $\mathcal{J}: \mathcal{G} \rightarrow L_{\mathrm{op}}(\mathcal{B}, \mathit{E})$.
    \item[ii)] For all $x \in \Omega$ and $f \in \mathcal{B}$, we have $f(x) = \mathcal{J} \left[K(x, \cdot)\right](f)$.
\end{description}

If, in addition, the following conditions also hold:
\begin{description}
    \item[iii)] For all $y \in \Omega'$, the function $K(\cdot, y)$ belongs to $\mathcal{B}$.
    \item[iv)] For all $g \in \mathcal{G}$, we have $g(y) = (\mathcal{J} \circ g)(K(\cdot, y))$, identifying $g$ with an element of $L_{\mathrm{op}}(\mathcal{B}, \mathit{E})$.
\end{description}

Then $\mathcal{B}$ is called a \textit{right-sided GRKBS}, and $K$ its \textit{right-sided reproducing kernel}.

If $\mathcal{B}$ satisfies both sets of conditions, it is referred to as a \textit{two-sided GRKBS}, and $K$ is called its \textit{two-sided reproducing kernel}, or simply its \textit{reproducing kernel}.
\end{definition}
\begin{remark}\label{remark1}
Note that if $\mathcal{B}$ is a left-sided GRKBS, then for all $x \in \Omega$, the evaluation operators $ev_{x}$ are continuous, i.e., $ev_{x} \in L_{\mathrm{op}}(\mathcal{B}, \mathit{E})$ for all $x \in \Omega$. Indeed,
\begin{align*}
\|f(x)\|_{\mathit{E}} = \|\mathcal{J}[\mathcal{K}(x, \cdot)](f)\|_{\mathit{E}} \leq \|\mathcal{J}[\mathcal{K}(x, \cdot)]\|_{L_{\mathrm{op}}(\mathcal{B}, \mathit{E})} \|f\|_{\mathcal{B}} = \|\mathcal{K}(x, \cdot)\|_{\mathcal{G}} \|f\|_{\mathcal{B}}.
\end{align*}
Thus, every left-sided GRKBS is, in fact, a GRKBS. In this setting, a configuration space $\mathcal{F}$ can be taken as $\mathcal{B}$ itself, and the characteristic map $\phi: \Omega \rightarrow L_{\mathrm{op}}(\mathcal{B}, \mathit{E})$ is given by $\phi(x) = \mathcal{J}\left[\mathcal{K}(x, \cdot)\right] \in L_{\mathrm{op}}(\mathcal{B}, \mathit{E})$. These GRKBS are associated with a function that plays the role of a reproducing kernel, establishing a strong analogy with reproducing kernel Hilbert spaces (RKHS). 
\end{remark}
Given a GRKBS, it is possible to construct a generalized two-sided reproducing kernel from a generating pair of the GRKBS. In other words, according to Definition~\ref{def6}, a GRKBS is a Banach space with a two-sided generalized reproducing kernel.   

We now state the following theorem.
\begin{theorem}\label{Teo2}
If $\mathcal{B}$ is a GRKBS with generating structure $\left(\mathcal{X}, \mathit{E}, \mathcal{F}, \phi\right)$, then it is a two-sided GRKBS, where the reproducing kernel can be defined as $\mathcal{K}: \mathcal{X} \times \mathcal{F}/\mathcal{N}_{\phi} \rightarrow \mathit{E}$, defined by $\mathcal{K}\left(x, \mu + \mathcal{N}\right) = \phi_{*}\left(\mu + \mathcal{N}_{\phi}\right)(x) = f_{\mu}(x) = \phi(x)(\mu) \in \mathit{E}$.
\end{theorem}
\begin{proof}
First, let us prove that $\mathcal{K}$ is a left-sided generalized reproducing kernel. For this purpose, we set $\mathcal{G} = L_{\mathrm{op}}\left(\mathcal{F}/\mathcal{N}_{\phi}, E\right)$, so that, according to the definition of $\mathcal{K}$, the condition $\mathcal{K}(x, \cdot) \in \mathcal{G}$ is satisfied.

In the proof of Proposition~\ref{prop4}, we defined $\widetilde{\phi}'_{*}$ as an isometric isomorphism from $L(\mathcal{F}/\mathcal{N}_{\phi}, E)$ to $L(\mathcal{B}, E)$, such that
\begin{align*}
\widetilde{\phi}'_{*}\left[\mathcal{K}(x, \cdot)\right](f) = \mathcal{K}(x, \phi^{-1}_{*}(f)) = \phi_{*}(\phi^{-1}_{*}(f))(x) = f(x).
\end{align*}
Accordingly, in order to obtain conditions \textbf{i)} and \textbf{ii)} of the theorem, we set $\mathcal{J} = \widetilde{\phi}'_{*}$. It follows that $\mathcal{K}$ is a left-sided generalized reproducing kernel.

We will now prove that $\mathcal{K}$ is also a right-sided generalized reproducing kernel. Note that for $\mu + \mathcal{N}_{\phi} \in \mathcal{F}/\mathcal{N}_{\phi}$, we have
\[
\mathcal{K}(\cdot, \mu + \mathcal{N}_{\phi}) = \phi(\cdot)(\mu) = f_{\mu}(\cdot) \in \mathcal{B}.
\]

Moreover, considering $g \in L(\mathcal{F}/\mathcal{N}_{\phi}, \mathit{E})$, we compute
\[
\widetilde{\phi}'_{*} \circ g\left(\mathcal{K}(\cdot, \mu + \mathcal{N}_{\phi})\right) =  g\left(\phi^{-1}_{*}(\mathcal{K}(\cdot, \mu + \mathcal{N}_{\phi}))\right) = g(\phi^{-1}_{*}(f_{\mu})) = g(\mu + \mathcal{N}_{\phi}).
\]

\end{proof}

This result unifies two notions of RKBS that were previously introduced and studied independently in the literature. Moreover, we have generalized these notions to the case where the space $\mathcal{B}$ consists of functions $f: \mathcal{X} \rightarrow \mathit{E}$, with $\mathit{E}$ an arbitrary Banach space. The classical case corresponds to $\mathit{E} = \mathbb{R}$. 

In what follows, we explore how this generalization enriches the study of such spaces. Throughout, we adopt the definition of GRKBS that follows from Proposition~\ref{prop2}.

Theorem~\ref{Teo2} facilitates the construction of a reproducing kernel that is uniquely determined by the configuration $(\mathcal{X}, \mathit{E}, \mathcal{F}, \phi)$. In this setting, we identify the map $K: \mathcal{X} \times \mathcal{F}/\mathcal{N}_{\phi} \rightarrow \mathit{E}$, constructed in the proof of Theorem~\ref{Teo2}, as the reproducing kernel of $\mathcal{B}$.

For instance, in Section \ref{secc1}, we described an example of a GRKBS, $\mathcal{B}_{\mathcal{N}\mathcal{N}}$, with structure $(\mathbb{R}^n, \mathbb{R}, \mathcal{M}(\Theta), \phi)$, where $\Theta = \mathbb{R}^n \times \mathbb{R}$ and $\phi$ are specified by the expression given in \eqref{phiBN}. In this case, the reproducing kernel is the function $K: \mathbb{R}^{n} \times \mathcal{M}(\Theta)/\mathcal{N}_{\phi} \rightarrow \mathbb{R}$, defined by
$$K(x, \mu + \mathcal{N}_{\phi}) = \int_{\Theta} \rho(x, \theta) \beta(\theta) \, d\mu(\theta).$$
Here, $K(x, \cdot)$ is identified with an element of the dual space $\mathcal{B}_{\mathcal{N}\mathcal{N}}'$, such that
$$\prescript{}{\mathcal{B}_{\mathcal{N}\mathcal{N}}'}{\left\langle K(x, \cdot), f \right\rangle}_{\mathcal{B}_{\mathcal{N}\mathcal{N}}} = \int_{\Theta} \rho(x, \theta) \beta(\theta) \, d\mu(\theta), \quad \forall f \in \mathcal{B}_{\mathcal{N}\mathcal{N}}, \, f = f_{\mu}.$$
Moreover, for $\mu \in \mathcal{M}(\Theta)$, we have $K(\cdot, \mu) \in \mathcal{B}_{\mathcal{N}\mathcal{N}}$, such that for all $g \in \mathcal{B}_{\mathcal{N}\mathcal{N}}'$,
$$\prescript{}{\mathcal{B}_{\mathcal{N}\mathcal{N}}'}{\left\langle g, f \right\rangle}_{\mathcal{B}_{\mathcal{N}\mathcal{N}}} = \prescript{}{\mathcal{B}_{\mathcal{N}\mathcal{N}}'}{\left\langle g, K(\cdot, \mu) \right\rangle}_{\mathcal{B}_{\mathcal{N}\mathcal{N}}},$$
where $f = f_{\mu}$.

In the remainder of this section, we will present several results concerning the uniqueness of the structure associated with a GRKBS.
\subsection{Some Remarks and Results on Uniqueness}
Up to this point, we have established the notion of a GRKBS. As shown in Proposition~\ref{prop2}, each GRKBS $\mathcal{B}$ is uniquely determined by a configuration $(\mathcal{X}, \mathit{E}, \mathcal{F}, \phi)$ that characterizes the space. Unless otherwise stated, we will focus solely on the pair $(\mathcal{F}, \phi)$, which we refer to as the \textit{configuration pair} of the GRKBS.
\begin{remark}\label{remark2}
It is important to highlight that, given $\mathcal{X}$ and $\mathit{E}$, a GRKBS can be generated by configuration pairs that are essentially different. In fact, suppose that $\mathcal{B}$ is a GRKBS with a generator pair $(\mathcal{F}, \phi)$ such that the closed subspace $\mathcal{N}_{\phi}$, defined in \eqref{defN}, is nontrivial, i.e., $\mathcal{N}_{\phi} \neq \{0\}$.

Then, the pair $(\mathcal{B}, \phi_{\mathrm{triv}})$, where $\phi_{\mathrm{triv}}: \mathcal{X} \rightarrow L_{\mathrm{op}}(\mathcal{B}, \mathit{E})$ is defined by $\phi_{\mathrm{triv}}(x) = \mathrm{ev}_x$, also constitutes a configuration for $\mathcal{B}$. This follows from Proposition~\ref{prop3}, noting that for the configuration pair $(\mathcal{B}, \phi_{\mathrm{triv}})$ we have $\mathcal{B} \cong \mathcal{B}/\mathcal{N}_{\phi_{\mathrm{triv}}}$, with $\mathcal{N}_{\phi_{\mathrm{triv}}} = \{0\}$.

We refer to the pair $(\mathcal{B}, \phi_{\mathrm{triv}})$ as the \emph{trivial configuration} associated with the GRKBS.

Note that $\mathcal{F}$ is not isometrically isomorphic to $\mathcal{B}$, since $\mathcal{N}_{\phi} \neq \{0\}$. We may thus conclude that, under the conditions discussed above, a GRKBS may admit multiple \emph{distinct} configurations, including at least the pair $(\mathcal{F}, \phi)$ and the trivial configuration $(\mathcal{B}, \phi_{\mathrm{triv}})$. The latter, however, does not provide any additional information beyond what is already contained in the space $\mathcal{B}$ itself.  
\end{remark}
This observation motivates the definition of when two configurations of a GRKBS are considered equivalent.

Before, let us present the following result.
\begin{theorem}\label{teo1}
Consider $\mathcal{B}$ as a GRKBS that possesses a configuration $(\mathcal{F}, \phi)$. Assume, furthermore, that $\widetilde{\mathcal{F}}$ is another Banach space such that
\[
\widetilde{\mathcal{F}} \cong \mathcal{F}.
\]
Then, we have $\widetilde{\mathcal{F}}/\widetilde{\mathcal{N}} \cong \mathcal{B}$, where
\[
\widetilde{\mathcal{N}} = \left\{ \widetilde{\mu} \in \widetilde{\mathcal{F}} : \phi(x)(\mathcal{J}\widetilde{\mu}) = 0,\ \forall x \in \mathcal{X} \right\},
\]
and $\mathcal{J}$ denotes the isometric isomorphism from $\widetilde{\mathcal{F}}$ to $\mathcal{F}$.

In other words, the pair $(\widetilde{\mathcal{F}}, \widetilde{\phi})$, where $\widetilde{\phi}: \mathcal{X} \rightarrow L_{\mathrm{op}}(\widetilde{\mathcal{F}}, \mathit{E})$ is defined by
\[
\widetilde{\phi}(x)(\widetilde{\mu}) = \phi(x)(\mathcal{J}\widetilde{\mu}), \quad \forall (x, \widetilde{\mu}) \in \mathcal{X} \times \widetilde{\mathcal{F}},
\]
also constitutes a configuration of $\mathcal{B}$.
\end{theorem}
\begin{proof}
Initially, observe that it is straightforward to verify that $\widetilde{\phi}$ defines a map from $\mathcal{X}$ to $L_{\mathrm{op}}(\widetilde{\mathcal{F}}, \mathit{E})$. The isometric isomorphism $\mathcal{J}$ induces an isometry between the closed subspaces $\widetilde{\mathcal{N}}$ and
\[
\mathcal{N} = \left\{\mu \in \mathcal{F} : \phi(x)(\mu) = 0,\ \forall x \in \mathcal{X}\right\},
\]
such that $\mathcal{N} = \mathcal{J}(\widetilde{\mathcal{N}})$. Consequently, $\mathcal{J}$ induces an isomorphism, also denoted by $\mathcal{J}$, between the quotient spaces $\widetilde{\mathcal{F}}/\widetilde{\mathcal{N}}$ and $\mathcal{F}/\mathcal{N}$, defined by
\[
\mathcal{J}(\widetilde{\mu} + \widetilde{\mathcal{N}}) = \mathcal{J}(\widetilde{\mu}) + \mathcal{N}, \quad \forall\ \widetilde{\mu} \in \widetilde{\mathcal{F}}.
\]

We now show that this isomorphism is, in fact, an isometry. Since $\mathcal{J}$ is an isometry from $\widetilde{\mathcal{F}}$ to $\mathcal{F}$ and $\mathcal{N} = \mathcal{J}(\widetilde{\mathcal{N}})$, it follows that
\[
\left\{ \| \bar{\mu} + \mathcal{J}(\widetilde{\mu}) \|_{\mathcal{F}} : \bar{\mu} \in \mathcal{N} \right\} = \left\{ \| \mu + \widetilde{\mu} \|_{\widetilde{\mathcal{F}}} : \mu \in \widetilde{\mathcal{N}} \right\}, \quad \forall\ \widetilde{\mu} \in \widetilde{\mathcal{F}}.
\]
Hence, these two sets of real numbers are equal, and thus their infima coincide. That is,
\[
\| \mathcal{J}(\widetilde{\mu}) + \mathcal{N} \|_{\mathcal{F}/\mathcal{N}} = \inf_{\bar{\mu} \in \mathcal{N}} \| \bar{\mu} + \mathcal{J}(\widetilde{\mu}) \|_{\mathcal{F}} = \inf_{\mu \in \widetilde{\mathcal{N}}} \| \mu + \widetilde{\mu} \|_{\widetilde{\mathcal{F}}} = \| \widetilde{\mu} + \widetilde{\mathcal{N}} \|_{\widetilde{\mathcal{F}}/\widetilde{\mathcal{N}}}.
\]

Therefore, we have established that
\[
\widetilde{\mathcal{F}}/\widetilde{\mathcal{N}} \cong \mathcal{F}/\mathcal{N} \cong \mathcal{B},
\]
thereby completing the proof.
\end{proof}

This result establishes a relationship between two configurations of the same GRKBS. In this context, Theorem~\ref{teo1} motivates the following definition.

\begin{definition}
Let $(\mathcal{F}, \phi)$ and $(\widetilde{\mathcal{F}}, \widetilde{\phi})$ be two configurations of a GRKBS $\mathcal{B}$. We say that these pairs are \emph{equivalent}, denoted by $(\mathcal{F}, \phi) \equiv (\widetilde{\mathcal{F}}, \widetilde{\phi})$, if and only if
\[
\mathcal{F} \cong \widetilde{\mathcal{F}} \quad \text{and} \quad \widetilde{\phi}(x)(\widetilde{\mu}) = \phi(x)(\mathcal{J}\widetilde{\mu}), \quad \forall\, x \in \mathcal{X},\; \widetilde{\mu} \in \widetilde{\mathcal{F}},
\]
where $\mathcal{J}: \widetilde{\mathcal{F}} \rightarrow \mathcal{F}$ is an isometric isomorphism.
\end{definition}

It is straightforward to verify that $\equiv$ defines an equivalence relation, which partitions the set of all configurations of a GRKBS into equivalence classes. As mentioned earlier, if $(\mathcal{F}, \phi)$ is a configuration of $\mathcal{B}$ such that $\mathcal{N}_{\phi} \neq \{0\}$, then the quotient set contains at least two distinct equivalence classes: the class of $(\mathcal{F}, \phi)$, and the trivial class, whose representative is $(\mathcal{B}, \phi_{\mathrm{triv}})$.

We now proceed to prove that, although a GRKBS may admit multiple equivalence classes of configurations, the significance of the reproducing kernel lies in its uniqueness—understood in a specific sense. This phenomenon is analogous to what occurs in the context of Reproducing Kernel Hilbert Spaces (RKHS) and Reproducing Kernel Banach Spaces (RKBS), as discussed in \cite{Xu2019}.

\begin{theorem}\label{teo3}
Let $\mathcal{B}$ be a GRKBS with two configurations $(\mathcal{F}_1, \phi_1)$ and $(\mathcal{F}_2, \phi_2)$. Then, the following statements hold:
\begin{itemize}
\item The pairs $\left(\mathcal{F}_1/\mathcal{N}_{\phi_{\mathcal{K}_1}}, \phi_{\mathcal{K}_1}\right)$ and $\left(\mathcal{F}_2/\mathcal{N}_{\phi_{\mathcal{K}_2}}, \phi_{\mathcal{K}_2}\right)$,  
where
\[
\phi_{\mathcal{K}_i}: \mathcal{X} \rightarrow L\left(\mathcal{F}_i/\mathcal{N}_{\phi_{\mathcal{K}_i}}, \mathit{E}\right) \quad \text{are defined by} \quad \phi_{\mathcal{K}_i}(x) = \mathcal{K}_i(x, \cdot),
\]
and where $\mathcal{K}_i$ for $i = 1, 2$ are the generalized reproducing kernels associated with each configuration, both belong to the trivial equivalence class.

\item In particular, this implies that for all $\mu_1 + \mathcal{N}_{\phi_{\mathcal{K}_1}} \in \mathcal{F}_1/\mathcal{N}_{\phi_{\mathcal{K}_1}},$
\[
\mathcal{K}_1\left(x, \mu_1 + \mathcal{N}_{\phi_{\mathcal{K}_1}}\right) = \mathcal{K}_2\left(x, \phi^{(2)-1}_{*} \circ \phi^{(1)}_{*} \left(\mu_1 + \mathcal{N}_{\phi_{\mathcal{K}_1}}\right)\right), \quad \forall x \in \mathcal{X},
\]
where $\phi^{(1)}_{*}$ and $\phi^{(2)}_{*}$ denote the isometries from $\mathcal{F}_1/\mathcal{N}_{\phi_{\mathcal{K}_1}}$ and $\mathcal{F}_2/\mathcal{N}_{\phi_{\mathcal{K}_2}},$ respectively, to $\mathcal{B}$.
\end{itemize}
\end{theorem} 
\begin{proof}
As is well known, if $(\mathcal{F}_1, \phi_1)$ and $(\mathcal{F}_2, \phi_2)$ are configurations of a GRKBS $\mathcal{B}$, then $\mathcal{F}_1/\mathcal{N}_{\phi_{\mathcal{K}_1}} \cong \mathcal{F}_2/\mathcal{N}_2$, since both are isometrically isomorphic to $\mathcal{B}$. In this case, the corresponding isometry is given by $\phi^{(2)-1}_{*} \circ \phi^{(1)}_{*}$. 

Accordingly, for every $\mu_1 + \mathcal{N}_{\phi_{\mathcal{K}_1}} \in \mathcal{F}_1/\mathcal{N}_{\phi_{\mathcal{K}_1}}$, there exists $\mu_2 + \mathcal{N}_{\phi_{\mathcal{K}_2}} \in \mathcal{F}_2/\mathcal{N}_{\phi_{\mathcal{K}_2}}$ such that
\[
\phi^{(2)-1}_{*} \circ \phi^{(1)}_{*} \left(\mu_1 + \mathcal{N}_{\phi_{\mathcal{K}_1}}\right) = \mu_2 + \mathcal{N}_{\phi_{\mathcal{K}_2}} \;\; \Rightarrow \;\; \phi^{(1)}_{*}\left(\mu_1 + \mathcal{N}_{\phi_{\mathcal{K}_1}}\right) = \phi^{(2)}_{*}\left(\mu_2 + \mathcal{N}_{\phi_{\mathcal{K}_2}}\right) \;\; \Rightarrow \;\; f_{\mu_1} = f_{\mu_2}.
\]

Taking this into account and applying the reproducing property of the generalized reproducing kernel, we obtain:
\begin{align*}
\mathcal{K}_2\left(x, \phi^{(2)-1}_{*} \circ \phi^{(1)}_{*}\left(\mu_1 + \mathcal{N}_{\phi_{\mathcal{K}_1}}\right)\right) 
&= \mathcal{K}_2\left(x, \mu_2 + \mathcal{N}_{\phi_{\mathcal{K}_2}}\right) = f_{\mu_2}(x) \\
&= f_{\mu_1}(x) = \mathcal{K}_1\left(x, \mu_1 + \mathcal{N}_{\phi_{\mathcal{K}_1}}\right), \quad \forall x \in \mathcal{X},\; \mu_1 + \mathcal{N}_{\phi_{\mathcal{K}_1}} \in \mathcal{F}_1/\mathcal{N}_{\phi_{\mathcal{K}_1}}.
\end{align*}

Through this equality, we have established the second point of the theorem. Moreover, it can be equivalently written as:
\[
\phi_{\mathcal{K}_2}(x)\left(\phi^{(2)-1}_{*} \circ \phi^{(1)}_{*}\left(\mu_1 + \mathcal{N}_{\phi_{\mathcal{K}_1}}\right)\right) = \phi_{\mathcal{K}_1}(x)\left(\mu_1 + \mathcal{N}_{\phi_{\mathcal{K}_1}}\right),
\]
which proves the first point of the theorem. 
\end{proof}

It should be noted that this result implies that, up to isometries, the generalized reproducing kernel associated with the corresponding GRKBS is unique. On the other hand, the generalized reproducing kernel is associated with generating pairs that belong to the trivial class; however, any generating pair $(\mathcal{F}, \phi)$ such that $\mathcal{N}_{\phi} \neq 0$ belongs to a different class. That is, the definition of GRKBS based on generating pairs is broader than that which merely considers reproducing kernels, as in \cite{Xu2019}.

In the next theorem, we will provide a characterization of the generator pairs of a GRKBS that belong to the same equivalence class.
\begin{theorem}\label{teo4}
Let $\left(\mathcal{F}_1, \phi_1\right)$ and $\left(\mathcal{F}_2, \phi_2\right)$ be two configurations of a GRKBS $\mathcal{B}$ such that $\mathcal{F}_1 \cong \mathcal{F}_2$. These configurations belong to the same class if and only if the following two conditions hold:
\begin{itemize}
\item $\mathcal{J}\mathcal{N}_{\phi_1} = \mathcal{N}_{\phi_2}$, where $\mathcal{J}$ is the isometric isomorphism between $\mathcal{F}_1$ and $\mathcal{F}_2$.
\item The following diagram is commutative:
\begin{align*}
\begin{array}{ccc}
\mathcal{F}_1 / \mathcal{N}_{\phi_1} & \xrightarrow{\mathcal{J}^{(c)}} & \mathcal{F}_2 / \mathcal{N}_{\phi_2}  \\
 \phi^{(1)}_{*} \downarrow  &  & \downarrow \phi^{(2)}_{*} \\
\mathcal{B} & \xrightarrow{I_d} & \mathcal{B}
\end{array}
\end{align*}
that is, $\phi^{(2)}_{*} \circ \mathcal{J}^{(c)} \circ \phi^{(1)-1}_{*} = I_d$, where $\mathcal{J}^{(c)}: \mathcal{F}_1 / \mathcal{N}_{\phi_1} \rightarrow \mathcal{F}_2 / \mathcal{N}_{\phi_2}$ is the isometry induced by $\mathcal{J}$, such that $\mathcal{J}^{(c)}\left(\mu_1 + \mathcal{N}_{\phi_1}\right) = \mathcal{J}(\mu_1) + \mathcal{N}_{\phi_2}$.
\end{itemize}
\end{theorem}
\begin{proof}
First, let us assume that the two conditions of the theorem hold. Since $\mathcal{J}\mathcal{N}_{\phi_1} = \mathcal{N}_{\phi_2}$, the isomorphism $\mathcal{J}^{(c)}: \mathcal{F}_1 / \mathcal{N}_{\phi_1} \rightarrow \mathcal{F}_2 / \mathcal{N}_{\phi_2}$ induced by $\mathcal{J}$ is an isometry. Indeed, we have
\begin{align*}
\|\mathcal{J}^{(c)}\left(\mu_1 + \mathcal{N}_{\phi_1}\right)\|_{\mathcal{F}_2 / \mathcal{N}_{\phi_2}} &= \inf_{\nu_2 \in \mathcal{N}_{\phi_2}}\|\mathcal{J}(\mu_1) + \nu_2\|_{\mathcal{F}_2} = \inf_{\nu_1 \in \mathcal{N}_{\phi_1}}\|\mathcal{J}(\mu_1 + \nu_1)\|_{\mathcal{F}_2} \\
&= \inf_{\nu_1 \in \mathcal{N}_{\phi_1}}\|\mu_1 + \nu_1\|_{\mathcal{F}_1} = \|\mu_1 + \mathcal{N}_{\phi_1}\|_{\mathcal{F}_1 / \mathcal{N}_{\phi_1}}.
\end{align*}
In the above, we used the fact that $\mathcal{J}\mathcal{N}_{\phi_1} = \mathcal{N}_{\phi_2}$ for the second equality and that $\mathcal{J}$ is an isometry for the third equality. 

On the other hand, let $\mu_1 \in \mathcal{F}_1$ and $\mu_2 \in \mathcal{F}_2$ such that $\mu_2 = \mathcal{J}\mu_1$. Then, the following chain of equivalences holds:
\begin{align*}
&\mu_2 + \mathcal{N}_{\phi_{\mathcal{K}_2}} = \mathcal{J}^{(c)}\left(\mu_1 + \mathcal{N}_{\phi_1}\right) = \phi^{(2)-1}_{*}\circ \phi^{(1)}_{*}\left(\mu_1 + \mathcal{N}_{\phi_1}\right) \\
&\Leftrightarrow \; \phi^{(2)}_{*}\left(\mu_2 + \mathcal{N}_2\right) = \phi^{(1)}_{*}\left(\mu_1 + \mathcal{N}_{\phi_1}\right) \Leftrightarrow \; f_{\mu_1} = f_{\mu_2} \Leftrightarrow \; f_{\mu_2}(x) = f_{\mu_1}(x), \forall x \in \mathcal{X} \\
&\Leftrightarrow \; \phi_2(x)(\mu_2) = \phi_1(x)(\mu_1), \forall x \in \mathcal{X}.
\end{align*}
Thus, taking into account the second condition of the theorem and that $\mu_2 = \mathcal{J}\mu_1$, the last equality implies that 
\begin{align*}
\phi_1(x)(\mu_1) = \phi_2(x)(\mathcal{J}\mu_1), \forall x \in \mathcal{X}, \mu_1 \in \mathcal{F}_1.
\end{align*}
Since we have assumed in the theorem's hypothesis that $\mathcal{F}_1 \cong \mathcal{F}_2$, the last equality ultimately proves that the generator pairs $(\mathcal{F}_1, \phi_1)$ and $(\mathcal{F}_2, \phi_2)$ belong to the same equivalence class. 

Now, we need to prove that if $(\mathcal{F}_1, \phi_1)$ and $(\mathcal{F}_2, \phi_2)$ are in the same equivalence class, then the two conditions of the theorem hold. 

Let us see: if $\mu_1 \in \mathcal{N}_{\phi_1}$, then 
\[
\phi_1(x)(\mu_1) = \phi_2(x)(\mathcal{J}\mu_1) = 0,
\]
which means that $\mathcal{J}\mu_1 \in \mathcal{N}_2$. We have shown that $\mathcal{J}\mathcal{N}_{\phi_1} \subset \mathcal{N}_2$. 

Similarly, if $\mu_2 \in \mathcal{N}_2$, there exists $\mu_1 \in \mathcal{F}_1$ such that $\mu_2 = \mathcal{J}\mu_1$. Moreover, given the equivalence of the generator pairs, we have
\[
\phi_1(x)(\mu_1) = \phi_2(x)(\mathcal{J}\mu_1) = \phi_2(x)(\mu_2) = 0 \;\; \Rightarrow \;\; \mu_1 \in \mathcal{N}_{\phi_1},
\]
which implies $\mathcal{N}_2 \subset \mathcal{J}\mathcal{N}_{\phi_1}$. We have now proven that $\mathcal{N}_2 = \mathcal{J}\mathcal{N}_{\phi_1}$. 

Finally, let $f \in \mathcal{B}$, $\mu_2 \in \mathcal{F}_2$, and $\mu_1 \in \mathcal{F}_1$ such that $\mu_2 = \mathcal{J}\mu_1$, $\left(\mu_2 + \mathcal{N}_2 = \mathcal{J}^{(c)}(\mu_1 + \mathcal{N}_{\phi_1})\right)$, and $f = f_{\mu_2}$. Given that the pairs $(\mathcal{F}_1, \phi_1)$ and $(\mathcal{F}_2, \phi_2)$ are equivalent, we have 
\[
f = f_{\mu_2} = f_{\mu_1}.
\]

From the above, we obtain 
\begin{align}\label{eqn4}
f_{\mu_1} = f_{\mu_2} \Leftrightarrow \phi^{(1)}_{*}(\mu_1 + \mathcal{N}_{\phi_1}) = \phi^{(2)}_{*}\left(\mu_2 + \mathcal{N}_2\right) = \phi^{(2)}_{*}(\mathcal{J}^{(c)}(\mu_1 + \mathcal{N}_{\phi_1})),
\end{align}
and since $f = \phi^{(1)}_{*}(\mu_1 + \mathcal{N}_{\phi_1}) \Leftrightarrow \mu_1 + \mathcal{N}_{\phi_1} = \phi^{(1)-1}_{*}(f)$, we obtain
\begin{align*}
f = \phi^{(2)}_{*}(\mu_2 + \mathcal{N}_2) = \phi^{(2)}_{*}(\mathcal{J}^{(c)}(\phi^{(1)-1}_{*}(f))), \forall f \in \mathcal{B},
\end{align*}
which implies that $\phi^{(2)}_{*} \circ \mathcal{J}^{(c)} \circ \phi^{(1)-1}_{*} = I_d$.
\end{proof}

Theorems~\ref{teo3} and~\ref{teo4} highlight the fact that the role played by configuration pairs is substantially broader than that of the reproducing kernel itself, which is associated with a specific class of generators: the trivial class. 

However, there exist generating pairs that do not belong to this class and, consequently, are not associated with any reproducing kernel. For instance, if $(\mathcal{F}, \phi)$ is a configuration such that $\mathcal{N}_{\phi} \neq 0$, then this pair does not belong to the trivial class—namely, the class corresponding to the generalized reproducing kernel.

\subsection{Abstract neural networks}\label{sec:AbsNN}
An interesting feature of the framework proposed in this work—one that also has significant implications for connecting these spaces with deep neural networks—is the ability to construct what we refer to as \emph{composition structures} from two or more GRKBSs.

The key idea is to translate into this abstract context the notion that deep architectures, composed of multiple layers, can be interpreted as a kind of composition of single-hidden-layer neural networks. To achieve this, it is necessary to construct functional structures by appropriately arranging multiple GRKBSs.

The main challenge lies in the fact that it is not straightforward to identify an operation between GRKBSs that yields another GRKBS. However, as we will demonstrate in this section, such an operation is not required. It is possible to define a composition operation involving two or more GRKBSs that results in a family of GRKBSs with properties analogous to those found in deep neural network architectures developed in the field of Machine Learning. 

Before proceeding with this section, we introduce a clarification regarding the notation that will be used henceforth: we denote by $\mathcal{B}(\mathcal{X}, E)$ a GRKBS consisting of functions $f: \mathcal{X} \rightarrow E$. Furthermore, if $(\mathcal{F}, \phi)$ is a configuration of $\mathcal{B}(\mathcal{X}, E)$, we refer to the isometric isomorphism between
\[
\mathcal{B}(\mathcal{X}, E) \;\; \text{and} \;\; \mathcal{F}/\mathcal{N}_{\phi}
\]
as the \emph{characteristic isometry}.

Let us consider two GRKBSs, $\mathcal{B}_0\left(\mathcal{X}_0, E_0\right)$ and $\mathcal{B}_1(E_0, E_1)$, with configurations $(\mathcal{F}_0, \phi_0)$ and $(\mathcal{F}_1, \phi_1)$, respectively. We say that these two spaces are \textbf{\textit{nested}} in the sense that the domain of the functions in $\mathcal{B}_1$ coincides with the codomain of the functions in $\mathcal{B}_0$. We aim to construct a structure from these two spaces. Consequently, a natural and logical starting point is to consider the space of composed functions $g \circ f: \mathcal{X}_0 \rightarrow E_1$, where $f \in \mathcal{B}_0\left(\mathcal{X}_0, E_0\right)$ and $g \in \mathcal{B}_1(E_0, E_1)$. 

However, this set does not form a GRKBS, since it is not possible to define a vector space structure on it using pointwise addition and scalar multiplication. In particular, if we consider two functions $g_0 \circ f_0$ and $g_1 \circ f_1$, where $f_0, f_1 \in \mathcal{B}_0\left(\mathcal{X}_0, E_0\right)$ and $g_0, g_1 \in \mathcal{B}_1\left(E_0, E_1\right)$, the pointwise sum $g_0 \circ f_0 + g_1 \circ f_1$ does not necessarily belong to the set in question.

To address this issue, we propose fixing the inner function in the composition and treating it as a parameter. In this sense, for every $\mu \in \mathcal{F}_0$, we can define the family of GRKBSs $\mathcal{B}_{\mu}(\mathcal{X}_0, E_1)$ with configuration $(\mathcal{F}_1, \phi_{\mu})$, where $\phi_{\mu}: \mathcal{X}_0 \rightarrow L_{\mathrm{op}}(\mathcal{F}_1, E_1)$ is given by
\[
\phi_{\mu}(x)(\nu) = \phi_1\left(\phi_0(x)(\mu)\right)(\nu), \quad \forall x \in \mathcal{X}_0,\; \nu \in \mathcal{F}_1,
\]
where $\phi_0$ and $\phi_1$ are the configuration maps of $\mathcal{B}_0\left(\mathcal{X}_0, E_0\right)$ and $\mathcal{B}_1\left(E_0, E_1\right)$, respectively. Observe that the elements of the family $\left\{\mathcal{B}_{\mu}(\mathcal{X}_0, E_1)\right\}_{\mu \in \mathcal{F}_0}$ are well defined, since $\phi_{\mu}(x)$ indeed belongs to $L_{\mathrm{op}}(\mathcal{F}_1, E_1)$, and satisfies
\[
\left\|\phi_{\mu}(x)\right\|_{L_{\mathrm{op}}(\mathcal{F}_1, E_1)} \leq \left\|\phi_1\left(\phi_0(x)(\mu)\right)\right\|_{L_{\mathrm{op}}(\mathcal{F}_1, E_1)}.
\]
Therefore, $(\mathcal{F}_1, \phi_{\mu})$ constitutes a valid configuration that makes $\mathcal{B}_{\mu}(\mathcal{X}_0, E_1)$ a GRKBS for every $\mu \in \mathcal{F}_0$.
\begin{definition}\label{def7}
Given two nested GRKBSs, $\mathcal{B}_0\left(\mathcal{X}_0, E_0\right)$ and $\mathcal{B}_1(E_0, E_1)$, with configurations $(\mathcal{F}_0, \phi_0)$ and $(\mathcal{F}_1, \phi_1)$, respectively, we say that the family of GRKBSs $\left\{\mathcal{B}_{\mu}(\mathcal{X}_0, E_1)\right\}_{\mu \in \mathcal{F}_0}$, constructed in the preceding paragraph, constitutes a \textbf{one-hidden-layer abstract neural network (AbsNN)}.

In this context, we shall refer to $\mathcal{X}_0$ as the input set and to $E_1$ as the output set.
\end{definition}

We can observe that Definition~\ref{def7} is, to some extent, incomplete, as it does not explicitly explain why the structure is referred to as a one-hidden-layer network. However, this will be clarified later. For now, we emphasize that the family of GRKBSs representing the abstract network is defined by a single parameter $\mu \in \mathcal{F}_0$. 

To illustrate this, let us examine how the spaces $\mathcal{B}_0\left(\mathcal{X}_0, E_0\right)$ and $\mathcal{B}_1(E_0, E_1)$ relate to the ABsNN they constitute. 

Let us consider the following: Given a GRKBS $\mathcal{B}(\mathcal{X}, E)$ with configuration $(\mathcal{F}, \phi)$ and a subset $\overline{\mathcal{X}} \subset \mathcal{X}$, we can construct a GRKBS of functions defined over $\overline{\mathcal{X}}$ using the configuration pair $(\mathcal{F}, \left.\phi\right|_{\overline{\mathcal{X}}})$. 

This new space, denoted by $\mathcal{B}\left(\overline{\mathcal{X}} \subset \mathcal{X}, E\right)$, is called the \emph{restriction} of $\mathcal{B}(\mathcal{X}, E)$ to $\overline{\mathcal{X}}$. 

In light of this, and considering the two nested GRKBSs discussed in the previous paragraphs, for each $\mu \in \mathcal{F}_0$, we define the subset $\mathcal{X}_{\mu} = \phi(\mathcal{X}_0)(\mu) \subset E_0$, and consider the restriction of $\mathcal{B}_1(E_0, E_1)$ to $\mathcal{X}_{\mu}$, denoted by $\mathcal{B}_1(\mathcal{X}_{\mu} \subset E_0, E_1)$.

We now state the following theorem.

\begin{theorem}\label{teo5}
For every $\mu \in \mathcal{F}_0$, the following Banach spaces are isometrically isomorphic:
\begin{align}\label{eqn5}
\frac{\mathcal{B}_1(E_0, E_1)}{\phi^{(1)}_{*}\left(\mathcal{N}_{\left.\phi_1\right|_{\mathcal{X}_{\mu}}} / \mathcal{N}_{\phi_1}\right)} 
\cong \mathcal{B}_1(\mathcal{X}_{\mu} \subset E_0, E_1) 
\cong \mathcal{B}_{\mu}(\mathcal{X}_0, E_1),
\end{align}
where $\phi^{(1)}_{*}$ denotes the characteristic isometry of $\mathcal{B}_1(E_0, E_1)$.
\end{theorem}

\begin{proof}
We first proceed to prove that $\mathcal{B}_1(\mathcal{X}_{\mu} \subset E_0, E_1) \cong \mathcal{B}_{\mu}(\mathcal{X}_0, E_1)$.

Let us show that $\mathcal{N}_{\phi_{\mu}} = \mathcal{N}_{\left.\phi_1\right|_{\mathcal{X}_{\mu}}}$. Indeed, suppose that $\nu \in \mathcal{N}_{\phi_{\mu}}$. Then it follows that
\[
\phi_{\mu}(x)(\nu) = \phi_1\left(\phi_0(x)(\mu)\right)(\nu) = 0, \quad \forall x \in \mathcal{X}_0.
\]
In particular, for every $e^{(\mu)} \in \mathcal{X}_{\mu}$ such that $e^{(\mu)} = \phi_0(x)(\mu)$ for some $x \in \mathcal{X}_0$, we have
\[
\phi_1\left(e^{(\mu)}\right)(\nu) = \phi_1(\phi_0(x)(\mu))(\nu) = 0,
\]
and therefore $\nu \in \mathcal{N}_{\left.\phi_1\right|_{\mathcal{X}_{\mu}}}$. This proves that $\mathcal{N}_{\phi_{\mu}} \subseteq \mathcal{N}_{\left.\phi_1\right|_{\mathcal{X}_{\mu}}}$.

Conversely, suppose that $\nu \in \mathcal{N}_{\left.\phi_1\right|_{\mathcal{X}_{\mu}}}$. Then for all $e^{(\mu)} \in \mathcal{X}_{\mu}$, we have $\phi_1(e^{(\mu)})(\nu) = 0$. In particular, for every $x \in \mathcal{X}_0$, and taking $e^{(\mu)} = \phi_0(x)(\mu)$, it follows that
\[
\phi_1\left(e^{(\mu)}\right)(\nu) = \phi_1(\phi_0(x)(\mu))(\nu) = \phi_{\mu}(x)(\nu) = 0.
\]
Since the above holds for all $x \in \mathcal{X}_0$, we conclude that $\nu \in \mathcal{N}_{\phi_{\mu}}$. Therefore, $\mathcal{N}_{\left.\phi_1\right|_{\mathcal{X}_{\mu}}} \subseteq \mathcal{N}_{\phi_{\mu}}$. If $\mathcal{N}_{\phi_{\mu}} = \mathcal{N}_{\left.\phi_1\right|_{\mathcal{X}_{\mu}}}$, then it follows that
\[
\mathcal{F}_1/\mathcal{N}_{\phi_{\mu}} = \mathcal{F}_1/\mathcal{N}_{\left.\phi_{1}\right|_{\mathcal{X}_{\mu}}} \Rightarrow 
\mathcal{B}_1(\mathcal{X}_{\mu} \subset E_0, E_1) \cong \mathcal{B}_{\mu}(\mathcal{X}_0, E_1).
\]
In the remaining part of the proof, we will establish the existence of an isometry between the spaces 
\[
\frac{\mathcal{B}_1(E_0, E_1)}{\phi^{(1)}_{*}\left(\mathcal{N}_{\left.\phi_1\right|_{\mathcal{X}_{\mu}}} / \mathcal{N}_{\phi_1}\right)} \quad \text{and} \quad \mathcal{B}_1(\mathcal{X}_{\mu} \subset E_0, E_1).
\]
To this end, observe that
\[
\mathcal{N}_{\phi_1} \subset \mathcal{N}_{\left.\phi_1\right|_{\mathcal{X}_{\mu}}} \subset \mathcal{F}_1,
\]
where the first two are closed subspaces of $\mathcal{F}_1$. Within this context, we may consider the quotient spaces
\[
\mathcal{F}_1/\mathcal{N}_{\phi_1},\;\; \mathcal{F}_1/\mathcal{N}_{\left.\phi_1\right|_{\mathcal{X}_{\mu}}},\;\; \left(\mathcal{N}_{\left.\phi_1\right|_{\mathcal{X}_{\mu}}}\right)/\mathcal{N}_{\phi_1},
\]
each of which is a Banach space endowed with the corresponding quotient norm. In particular, we have
\[
\left(\mathcal{N}_{\left.\phi_1\right|_{\mathcal{X}_{\mu}}}\right)/\mathcal{N}_{\phi_1} = \pi\left(\mathcal{N}_{\left.\phi_1\right|_{\mathcal{X}_{\mu}}}\right),
\]
where $\pi: \mathcal{F}_1 \rightarrow \mathcal{F}_1/\mathcal{N}_{\phi_1}$ denotes the canonical projection. Consequently, $\left(\mathcal{N}_{\left.\phi_1\right|_{\mathcal{X}_{\mu}}}\right)/\mathcal{N}_{\phi_1}$ is a closed subspace of $\mathcal{F}_1/\mathcal{N}_{\phi_1}$.

Taking all the above into account, we define the following map, which is clearly linear:
\begin{align*}
 \widetilde{\pi}: \frac{\mathcal{F}_1/\mathcal{N}_{\phi_1}}{\left(\mathcal{N}_{\left.\phi_1\right|_{\mathcal{X}_{\mu}}}\right)/\mathcal{N}_{\phi_1}} \rightarrow \mathcal{F}_1/\mathcal{N}_{\left.\phi_1\right|_{\mathcal{X}_{\mu}}},
\end{align*}
where, for every $\nu \in \mathcal{F}_1$, the correspondence is defined as
\[
\left(\nu + \mathcal{N}_{\phi_1}\right) + \left(\mathcal{N}_{\left.\phi_1\right|_{\mathcal{X}_{\mu}}}\right)/\mathcal{N}_{\phi_1} \mapsto \nu + \mathcal{N}_{\left.\phi_1\right|_{\mathcal{X}_{\mu}}}.
\]
We first verify that $\widetilde{\pi}$ is well-defined. Suppose $\nu_1, \nu_2 \in \mathcal{F}_1$ are such that
\[
\left(\nu_1 + \mathcal{N}_{\phi_1}\right) + \left(\mathcal{N}_{\left.\phi_1\right|_{\mathcal{X}_{\mu}}} / \mathcal{N}_{\phi_1}\right) = \left(\nu_2 + \mathcal{N}_{\phi_1}\right) + \left(\mathcal{N}_{\left.\phi_1\right|_{\mathcal{X}_{\mu}}} / \mathcal{N}_{\phi_1}\right),
\]
then
\[
(\nu_1 - \nu_2) + \mathcal{N}_{\phi_1} \in \mathcal{N}_{\left.\phi_1\right|_{\mathcal{X}_{\mu}}} / \mathcal{N}_{\phi_1} = \pi\left(\mathcal{N}_{\left.\phi_1\right|_{\mathcal{X}_{\mu}}}\right) \Rightarrow \nu_1 - \nu_2 \in \mathcal{N}_{\left.\phi_1\right|_{\mathcal{X}_{\mu}}} \Rightarrow \nu_1 + \mathcal{N}_{\left.\phi_1\right|_{\mathcal{X}_{\mu}}} = \nu_2 + \mathcal{N}_{\left.\phi_1\right|_{\mathcal{X}_{\mu}}}.
\]
This also shows that $\widetilde{\pi}$ is injective, since for any $\nu \in \mathcal{F}_1$,
\[
\left(\nu + \mathcal{N}_{\phi_1}\right) + \left(\mathcal{N}_{\left.\phi_1\right|_{\mathcal{X}_{\mu}}} / \mathcal{N}_{\phi_1}\right) = \left(0 + \mathcal{N}_{\phi_1}\right) + \left(\mathcal{N}_{\left.\phi_1\right|_{\mathcal{X}_{\mu}}} / \mathcal{N}_{\phi_1}\right) \Rightarrow \nu \in \mathcal{N}_{\left.\phi_1\right|_{\mathcal{X}_{\mu}}}.
\]

Furthermore, it is straightforward to show that $\widetilde{\pi}$ is surjective. Therefore, $\widetilde{\pi}$ is an isomorphism.

We now prove that $\widetilde{\pi}$ is an isometry. To that end, we compute:
\begin{align*}
&\left\|\left(\nu + \mathcal{N}_{\phi_1}\right) + \left(\mathcal{N}_{\left.\phi_1\right|_{\mathcal{X}_{\mu}}} / \mathcal{N}_{\phi_1}\right)\right\|_{\frac{\mathcal{F}_1/\mathcal{N}_{\phi_1}}{\left(\mathcal{N}_{\left.\phi_1\right|_{\mathcal{X}_{\mu}}}\right)/\mathcal{N}_{\phi_1}}} \\
&= \inf_{\widetilde{\nu} + \mathcal{N}_{\phi_1} \in \left(\mathcal{N}_{\left.\phi_1\right|_{\mathcal{X}_{\mu}}}\right)/\mathcal{N}_{\phi_1}} \left\|(\nu - \widetilde{\nu}) + \mathcal{N}_{\phi_1}\right\|_{\mathcal{F}_1/\mathcal{N}_{\phi_1}} \\
&= \inf_{\widetilde{\nu} \in \mathcal{N}_{\left.\phi_1\right|_{\mathcal{X}_{\mu}}}} \left\|(\nu - \widetilde{\nu}) + \mathcal{N}_{\phi_1}\right\|_{\mathcal{F}_1/\mathcal{N}_{\phi_1}} \\
&= \inf_{\substack{\widetilde{\nu} \in \mathcal{N}_{\left.\phi_1\right|_{\mathcal{X}_{\mu}}} \\ \bar{\nu} \in \mathcal{N}_{\phi_1}}} \left\| \nu - (\widetilde{\nu} + \bar{\nu}) \right\|_{\mathcal{F}_1} \\
&= \inf_{\widehat{\nu} \in \mathcal{N}_{\left.\phi_1\right|_{\mathcal{X}_{\mu}}}} \left\| \nu - \widehat{\nu} \right\|_{\mathcal{F}_1} = \left\| \nu + \mathcal{N}_{\left.\phi_1\right|_{\mathcal{X}_{\mu}}} \right\|_{\mathcal{F}_1/\mathcal{N}_{\left.\phi_1\right|_{\mathcal{X}_{\mu}}}}.
\end{align*}
This equality proves that $\widetilde{\pi}$ preserves norms, and therefore it is an isometry.

In conclusion, as established in Proposition~\ref{prop4}, there exist isometric isomorphisms
\begin{align*}
&\phi^{(1)}_{*}: \mathcal{F}_1/\mathcal{N}_{\phi_1} \rightarrow \mathcal{B}_1\left(E_0, E_1\right),\\
&\left(\left.\phi_1\right|_{\mathcal{X}_{\mu}}\right)_{*} : \mathcal{F}_1/\mathcal{N}_{\left.\phi_1\right|_{\mathcal{X}_{\mu}}} \rightarrow \mathcal{B}_1(\mathcal{X}_{\mu} \subset E_0, E_1),
\end{align*}
and through these, the isometry $\widetilde{\pi}$ induces an isometric isomorphism between
\[
\frac{\mathcal{B}_1(E_0, E_1)}{\phi^{(1)}_{*}\left(\mathcal{N}_{\left.\phi_1\right|_{\mathcal{X}_{\mu}}} / \mathcal{N}_{\phi_1}\right)} \quad \text{and} \quad \mathcal{B}_1(\mathcal{X}_{\mu} \subset E_0, E_1).
\]
This completes the proof of the theorem. 
\end{proof}

\begin{remark}
Observe that the isometric isomorphisms between the GRKBSs given in \eqref{eqn5} induce a mapping between the elements of the AbsNN, viewed as functions from the input set $\mathcal{X}_0$ to the output set $E_1$, and functions defined on $E_0$ with values in $E_1$, which carry information encoded by the functions in $\mathcal{B}_0(\mathcal{X}_0, E_0)$. 

In this context, it is natural to interpret the network as encoding the information transmitted from the input set $\mathcal{X}_0$ through the functions in 
\[
\frac{\mathcal{B}_1(E_0, E_1)}{\phi^{(1)}_{*}\left(\mathcal{N}_{\left.\phi_1\right|_{\mathcal{X}_{\mu}}} / \mathcal{N}_{\phi_1}\right)}.
\]
Taking this into account, we denote the chain of isometric isomorphisms in \eqref{eqn5} by
\begin{align}
\mathcal{HL}_{\mu} :  \mathcal{B}_{\mu}(\mathcal{X}_0, E_1) \rightarrow \mathcal{B}_1(\mathcal{X}_{\mu} \subset E_0, E_1) \rightarrow \frac{\mathcal{B}_1(E_0, E_1)}{\phi^{(1)}_{*}\left(\mathcal{N}_{\left.\phi_1\right|_{\mathcal{X}_{\mu}}} / \mathcal{N}_{\phi_1}\right)},
\end{align}
and we shall say that the family $\left\{\mathcal{HL}_{\mu}\right\}_{\mu \in \mathcal{F}_0}$ constitutes the unique hidden layer of this AbsNN.
\end{remark}

We aim to extend the definition of abstract neural network (AbsNN) that we have just established to the case where three or more GRKBSs are combined. As we shall see, this will allow us to construct AbsNNs with two or more hidden layers.

Suppose we are given three nested GRKBSs:
\begin{align}\label{eqn6}
\begin{array}{cc}
\mathcal{B}_0(\mathcal{X}_0, E_0), & \text{with configuration $(\mathcal{F}_0, \phi_0),$} \\
\mathcal{B}_1(E_0, E_1), & \text{with configuration $(\mathcal{F}_1, \phi_1),$}\\
\mathcal{B}_2(E_1, E_2), & \text{with configuration $(\mathcal{F}_2, \phi_2).$}
\end{array}
\end{align}
From these, we define the following biparametric family of GRKBSs:
\begin{align*}
\left\{\mathcal{B}_{(\mu_0, \mu_1)}\left(\mathcal{X}_0, E_2\right)\right\}_{(\mu_0, \mu_1) \in \mathcal{F}_0\times \mathcal{F}_1},\;\; \text{with configuration $(\mathcal{F}_2,\phi_{(\mu_0, \mu_1)}),$}
\end{align*}
where
\begin{align*}
 \phi_{(\mu_0, \mu_1)}(x) = \phi_2\left(\phi_1(\phi_0(x)(\mu_0))(\mu_1)\right) \in L_{\mathrm{op}}(\mathcal{F}_2, E_2).  
\end{align*}
We say that the family $\left\{\mathcal{B}_{(\mu_0, \mu_1)}\left(\mathcal{X}_0, E_2\right)\right\}_{(\mu_0, \mu_1) \in \mathcal{F}_0\times \mathcal{F}_1}$ constitutes a two-hidden-layer AbsNN. 
\begin{proposition}\label{prop5}
Let us consider three nested GRKBSs as those given in \eqref{eqn6}. For each pair $(\mu_0, \mu_1) \in \mathcal{F}_0 \times \mathcal{F}_1$, we define
\[
\mathcal{X}_{\mu_0} = \phi_0(\mathcal{X}_0)(\mu_0) \;\; \text{and} \;\; \mathcal{X}_{(\mu_0, \mu_1)} = \phi_1(\mathcal{X}_{\mu_0})(\mu_1) = \phi_1(\phi_0(\mathcal{X}_0)(\mu_0))(\mu_1).
\]
This leads to the following sequences of isometric isomorphisms:
\begin{align*}
\mathcal{HL}_{\mu_0} = \mathcal{B}_{\mu_0}(\mathcal{X}_0, E_1) \rightarrow \mathcal{B}_{1}(\mathcal{X}_{\mu_0} \subset E_0, E_1) \rightarrow \frac{\mathcal{B}_1(E_0, E_1)}{\phi^{(1)}_{*}\left(\mathcal{N}_{\left.\phi_1\right|_{\mathcal{X}_{\mu_0}}}/\mathcal{N}_{\phi_1}\right)},
\end{align*}
and
\begin{align*}
\mathcal{HL}_{(\mu_0, \mu_1)} = \mathcal{B}_{(\mu_0, \mu_1)}(\mathcal{X}_0, E_2) \rightarrow \mathcal{B}_{2}(\mathcal{X}_{(\mu_0, \mu_1)} \subset E_1, E_2) \rightarrow \frac{\mathcal{B}_2(E_1, E_2)}{\phi^{(2)}_{*}\left(\mathcal{N}_{\left.\phi_2\right|_{\mathcal{X}_{(\mu_0, \mu_1)}}}/\mathcal{N}_{\phi_2}\right)},
\end{align*}
where $\phi^{(1)}_{*}$ and $\phi^{(2)}_{*}$ denote the characteristic isometries of $\mathcal{B}_1(E_0, E_1)$ and $\mathcal{B}_2(E_1, E_2)$, respectively.
\end{proposition}
\begin{proof}
The proof of this proposition is identical to that of Theorem~\ref{teo5}.
\end{proof}

With reference to the composition structure associated with the three GRKBSs given in \eqref{eqn6}, we say that the family of GRKBSs $\left\{\mathcal{B}_{(\mu_0, \mu_1)}(\mathcal{X}_0, E_2)\right\}_{(\mu_0, \mu_1) \in \mathcal{F}_0 \times \mathcal{F}_1}$ defines an AbsNN with two hidden layers:
\[
\left\{\mathcal{HL}_{\mu_0}\right\}_{\mu_0 \in \mathcal{F}_0} \;\; \text{and} \;\; \left\{\mathcal{HL}_{(\mu_0, \mu_1)}\right\}_{(\mu_0, \mu_1) \in \mathcal{F}_0 \times \mathcal{F}_1}.
\]

Finally, we aim to generalize the concept of an AbsNN with $n$ hidden layers, obtained through the composition process described in the previous paragraphs from $n + 1$ nested GRKBSs.

Let us consider a sequence of $n+1$ nested GRKBSs, $\mathcal{B}_i(\mathcal{X}_i, E_i)$, with configurations $(\mathcal{F}_i, \phi_i)$ for $i = 0,\ldots,n$, where
\[
\mathcal{X}_i = E_{i-1}, \quad i = 1,\ldots,n,
\]
and each $E_i$ is a Banach space. From these, we define the following $n$-parametric family of GRKBSs:
\[
\left\{\mathcal{B}_{(\mu_1, \mu_2,\ldots,\mu_{n-1})}(\mathcal{X}_0, E_n)\right\}_{(\mu_1, \mu_2,\ldots,\mu_{n-1}) \in \mathcal{F}_{0}\times\mathcal{F}_{1}\times\ldots\times\mathcal{F}_{n-1}},
\]
where each GRKBS in the family, indexed by $(\mu_1, \mu_2,\ldots,\mu_{n-1}) \in \mathcal{F}_{0}\times\mathcal{F}_{1}\times\ldots\times\mathcal{F}_{n-1}$, is defined by the configuration $\left(\mathcal{F}_n, \phi_{(\mu_1, \mu_2,\ldots,\mu_{n-1})}\right)$, where
\[
\phi_{(\mu_1, \mu_2,\ldots,\mu_{n-1})}: \mathcal{X}_0 \rightarrow L_{\mathrm{op}}(\mathcal{F}_n, E_n)
\]
is the map we now proceed to define. 

For every $\nu \in \mathcal{F}_n$ and $x \in \mathcal{X}_0$, we define $\phi_{(\mu_0, \mu_1, \ldots,\mu_{n-1})}(x)(\nu)$ in a sequential manner. Given $(\mu_0, \mu_1, \ldots,\mu_{n-1}) \in \mathcal{F}_0 \times \mathcal{F}_1 \times \cdots \times \mathcal{F}_{n-1}$, we define:
\begin{align}\label{eqn7}
\left\{
\begin{array}{c}
e^{(\mu_0)} = \phi_0(x)(\mu_0) \in E_0,\\
e^{(\mu_0, \mu_1)} = \phi_1(e^{(\mu_0)})(\mu_1) = \phi_1(\phi_0(x)(\mu_0))(\mu_1) \in E_1,\\
\vdots\\
e^{(\mu_0, \mu_1, \ldots, \mu_i)} = \phi_i\left(e^{(\mu_0, \mu_1, \ldots, \mu_{i-1})}\right)(\mu_i) \in E_i,
\end{array}
\right. \quad \text{for all } i = 1,\ldots, n-1.
\end{align}
Thus, in \eqref{eqn7}, we have defined an iterated sequence of $n$ elements such that
\[
\phi_{(\mu_0, \mu_1, \ldots,\mu_{n-1})}(x)(\nu) = \phi_n\left(e^{(\mu_0, \mu_1, \ldots, \mu_{n-1})}\right)(\nu).
\]
Once the $n$-parametric family of GRKBSs representing the AbsNN has been defined, we now describe the $n$ hidden layers that compose it. The first layer is determined by the following triple of isometric isomorphisms:
\begin{align*}
\mathcal{HL}_{\mu_0}: \mathcal{B}_{\mu_0}(\mathcal{X}_0, E_1) \rightarrow \mathcal{B}_{1}(\mathcal{X}_{\mu_0} \subset E_0, E_1) \rightarrow \frac{\mathcal{B}_{1}(E_0, E_1)}{\phi^{(1)}_{*}\left(\mathcal{N}_{\left.\phi_1\right|_{\mathcal{X}_{\mu_0}}}/{\mathcal{N}_{\phi_1}}\right)},
\end{align*}
where $\mathcal{X}_{\mu_0} = \phi_0\left(\mathcal{X}_0\right)(\mu_0) \subset E_0$, for $\mu_0 \in \mathcal{F}_0$. On the other hand, setting
\[
\mathcal{X}_{(\mu_0, \mu_1)} = \phi_1\left(\mathcal{X}_{\mu_0}\right)(\mu_1),
\]
the second layer consists of the following isometric isomorphisms:
\begin{align*}
\mathcal{HL}_{(\mu_0, \mu_1)}: \mathcal{B}_{(\mu_0, \mu_1)}(\mathcal{X}_0, E_2) \rightarrow \mathcal{B}_{2}(\mathcal{X}_{(\mu_0, \mu_1)} \subset E_1, E_2) \rightarrow \frac{\mathcal{B}_{2}(E_1, E_2)}{\phi^{(2)}_{*}\left(\mathcal{N}_{\left.\phi_2\right|_{\mathcal{X}_{(\mu_0, \mu_1)}}}/{\mathcal{N}_{\phi_2}}\right)}.
\end{align*}
By iterating this process, we obtain the $i$-th hidden layer as:
\begin{align*}
\mathcal{HL}_{(\mu_0,\ldots,\mu_{i})}: \mathcal{B}_{(\mu_0,\ldots,\mu_{i})}(\mathcal{X}_0, E_i) \rightarrow \mathcal{B}_i(\mathcal{X}_{(\mu_0,\ldots,\mu_{i-1})} \subset E_{i-1}, E_i) \rightarrow \frac{\mathcal{B}_{i}(E_{i-1}, E_i)}{\phi^{(i)}_{*}\left(\mathcal{N}_{\left.\phi_i\right|_{\mathcal{X}_{(\mu_0,\ldots,\mu_{i-1})}}}/{\mathcal{N}_{\phi_i}}\right)},
\end{align*}
for $i = 3,\ldots,n$, thereby obtaining the remaining $n-2$ hidden layers.  
\section{Existence of Sparse Minimizers for the Abstract Training Problem in GRKBSs and AbsNNs}\label{sec:ATP}
\subsection{Sparse Minimizers for the ATP in GRKBSs}
In this section, our goal is to define and analyze what we refer to as the \emph{abstract training problem} (ATP) in the setting of a GRKBS. We focus on examining the existence of solutions that can be expressed as finite linear combinations of specific types of functions. We begin by studying the ATP defined on a GRKBS, and subsequently extend our analysis to the ATP in abstract neural networks (AbsNNs).

In the context of GRKBSs, this problem can be formulated as a variational problem in a Banach space, which is typically ill-posed. Therefore, studying the existence of solutions becomes essential. Following the approach proposed in \cite{Bredies2020}, we interpret the abstract training problem as a variational problem governed by a regularizer satisfying the conditions of Theorem 3.3 in \cite{Bredies2020}. These conditions ensure the existence of \emph{sparse solutions}, that is, solutions expressible as finite linear combinations of elements from the hypothesis space.

We provide the following definition.

\begin{definition}\label{def5}
Let $\mathcal{B}$ be a GRKBS of functions $f: \mathcal{X} \rightarrow \mathit{E}$, let $\mathcal{L}: \mathit{E} \times \mathit{E} \rightarrow \mathbb{R}$ be a loss function, and let $\left\{x_i\right\}^{N}_{i = 1} \subset \mathcal{X}$ and $\left\{y_i\right\}^{N}_{i = 1} \subset \mathit{E}$ denote the training data. The \emph{abstract training problem} (ATP) associated with $\mathcal{B}, \mathcal{L}$, and the training set $\left\{(x_i, y_i)\right\}^{N}_{i=1}$ is defined as the following variational problem:
\begin{align}\label{atpB}
\inf_{f \in \mathcal{B}}\left(\frac{1}{N}\sum_{i = 1}^N \mathcal{L}(y_i, f(x_i)) + \|f\|_{\mathcal{B}}\right).
\end{align}
\end{definition}
Given that $\mathcal{B}$ is a GRKBS, we can formulate the problem \eqref{atpB} in terms of the characteristic space $\mathcal{F}$ associated with $\mathcal{B}$. Indeed, the following proposition holds.
\begin{proposition}
Under the conditions of Definition \ref{def5}, we have that
\begin{align*}
\inf_{f \in \mathcal{B}}\left(\frac{1}{N}\sum^N_{i = 1}\mathcal{L}(y_i, f(x_i)) + \|f\|_{\mathcal{B}}\right) = \inf_{\mu \in \mathcal{F}}\left(\frac{1}{N}\sum^N_{i = 1}\mathcal{L}(y_i, f_{\mu}(x_i)) + \|\mu\|_{\mathcal{F}}\right).
\end{align*}
Moreover, if $\mu^{*}$ is a solution to the problem
\begin{align}\label{atpmu}
 \inf_{\mu \in \mathcal{F}}\left(\frac{1}{N}\sum^N_{i = 1}\mathcal{L}(y_i, f_{\mu}(x_i)) + \|\mu\|_{\mathcal{F}}\right),
\end{align}
then $f_{\mu^{*}}$ is a solution to the problem \eqref{atpB}.
\end{proposition}
\begin{proof}
The proof of this proposition follows essentially the same strategy as the proof of Proposition 3.7 in \cite{Bertolucci2023}. However, in our case, we formulate and prove the result in a general setting, whereas \cite{Bertolucci2023} considers only a particular case of RKBS.

The key idea is to use the isometric isomorphism between $\mathcal{F}/\mathcal{N}$ and $\mathcal{B}$.

We begin with the following chain of equalities:
\begin{align*}
\inf_{f \in \mathcal{B}}\left(\frac{1}{N}\sum_{i = 1}^N \mathcal{L}(y_i, f(x_i)) + \|f\|_{\mathcal{B}}\right) 
&= \inf_{\mu \in \mathcal{F}}\left(\frac{1}{N}\sum_{i = 1}^N \mathcal{L}(y_i, f_{\mu}(x_i)) + \|\mu + \mathcal{N}\|_{\mathcal{F}/\mathcal{N}}\right) \\
&= \inf_{\mu \in \mathcal{F}}\left(\frac{1}{N}\sum_{i = 1}^N \mathcal{L}(y_i, f_{\mu}(x_i)) + \inf_{\nu \in \mu + \mathcal{N}} \|\nu\|_{\mathcal{F}}\right) \\
&= \inf_{\mu \in \mathcal{F}, \, \nu \in \mu + \mathcal{N}} \left(\frac{1}{N}\sum_{i = 1}^N \mathcal{L}(y_i, f_{\mu}(x_i)) + \|\nu\|_{\mathcal{F}}\right) \\
&= \inf_{\mu \in \mathcal{F}} \left(\frac{1}{N}\sum_{i = 1}^N \mathcal{L}(y_i, f_{\mu}(x_i)) + \|\mu\|_{\mathcal{F}}\right).
\end{align*}

This proves the desired identity between the infima. Now, we show that $f_{\mu^*}$ solves problem \eqref{atpB} whenever $\mu^*$ solves problem \eqref{atpmu}.

By definition of minimizer, for all $\nu \in \mathcal{F}$,
\begin{align*}
\frac{1}{N}\sum_{i = 1}^N \mathcal{L}(y_i, f_{\mu^*}(x_i)) + \|\mu^*\|_{\mathcal{F}} 
\leq \frac{1}{N}\sum_{i = 1}^N \mathcal{L}(y_i, f_{\nu}(x_i)) + \|\nu\|_{\mathcal{F}}.
\end{align*}
Now, let us fix an arbitrary $\mu \in \mathcal{F}$ and choose $\nu \in \mathcal{F}$ such that $\nu \in \mu + \mathcal{N}_{\phi}$ and $f_{\nu} = f_{\mu}$. Since the previous inequality still holds for such $\nu$, and hence also for $\inf_{\nu \in \mu + \mathcal{N}_{\phi}} \|\nu\|_{\mathcal{F}}$, we obtain
\begin{align*}
\frac{1}{N}\sum_{i = 1}^N \mathcal{L}(y_i, f_{\mu^*}(x_i)) + \|\mu^*\|_{\mathcal{F}} 
&\leq \frac{1}{N}\sum_{i = 1}^N \mathcal{L}(y_i, f_{\mu}(x_i)) + \inf_{\nu \in \mu + \mathcal{N}_{\phi}}\|\nu\|_{\mathcal{F}}\\
&= \frac{1}{N}\sum_{i = 1}^N \mathcal{L}(y_i, f_{\mu}(x_i)) + \|f_{\mu}\|_{\mathcal{B}}.
\end{align*}
This inequality holds for all $\mu \in \mathcal{F}$, and thus for all $f_{\mu} \in \mathcal{B}$. In particular, taking $\mu = \mu^*$ yields
\[
\|\mu^*\|_{\mathcal{F}} \leq \|f_{\mu^*}\|_{\mathcal{B}},
\]
and therefore
\[
\frac{1}{N}\sum_{i = 1}^N \mathcal{L}(y_i, f_{\mu^*}(x_i)) + \|f_{\mu^*}\|_{\mathcal{B}} 
\leq \frac{1}{N}\sum_{i = 1}^N \mathcal{L}(y_i, f_{\mu}(x_i)) + \|f_{\mu}\|_{\mathcal{B}}, \quad \forall \mu \in \mathcal{F}.
\]
This proves that $f_{\mu^*}$ is a solution to problem \eqref{atpB}.
\end{proof}

In what follows, we focus on demonstrating the existence of at least one sparse solution to the problem \eqref{atpmu}. To do this, we will employ Theorem 3.3 from \cite{Bredies2020}, which establishes the conditions that must be satisfied by the problem \eqref{atpmu} for it to have a sparse solution. First, let us enumerate these conditions in our particular case.

We assume that $\mathcal{B}(\mathcal{X}, \mathit{E})$ is a GRKBS with configuration $\left(\mathcal{F}, \phi \right)$ satisfying the following conditions. For the characteristic space $\mathcal{F}$ endowed with the norm $\|\cdot\|_{\mathcal{F}},$ there exists a topology $\tau$ such that the topological space $(\mathcal{F}, \tau)$ is locally convex, and the sets 
\[
S^{-}\left(\|\cdot\|_{\mathcal{F}}, \alpha\right) := \left\{\mu \in \mathcal{F}: \|\mu\|_{\mathcal{F}} \leq \alpha\right\},
\]
are $\tau$-compact for all $\alpha > 0$. We say that $\|\cdot\|_{\mathcal{F}}$ is coercive with respect to the topology $\tau$ if it satisfies this latter condition. In addition, we assume that $\|\cdot\|_{\mathcal{F}}$ is lower semi-continuous with respect to the topology $\tau$.

A relatively natural way to obtain a space with these characteristics is to consider that $\mathcal{F}$ is isometrically $\mathcal{K}'$, the dual of a certain Banach space $\mathcal{K}$ with norm $\|\cdot\|_{\mathcal{K}}$ such that $$\|\mu\|_{\mathcal{F}} = \sup \left\{\prescript{}{\mathcal{K}'}{\left\langle \mu, \psi \right\rangle}_{\mathcal{K}}: \psi \in \mathcal{K}, \|\psi\|_{\mathcal{K}} \leq 1\right\},$$ for all $\mu \in \mathcal{F}.$ Then, we take $\tau = \sigma(\mathcal{K}', \mathcal{K})$, the weak* topology on $\mathcal{K}'$. Note that $\tau$ is the weak topology on $\mathcal{F}$. With this in mind, we have that $(\mathcal{F}, \tau)$ is locally convex. Indeed, let $\left\{\psi_1, \ldots, \psi_k\right\}$ be a finite subset of $\mathcal{K}$ and $\epsilon > 0$
\begin{align*}
V(\epsilon; \psi_1, \ldots, \psi_k) = \left\{\mu \in \mathcal{F}: \left|\prescript{}{\mathcal{K}'}{\left\langle \mu, \psi_i \right\rangle}_{\mathcal{K}}\right| < \epsilon, \forall i \in I \right\}.
\end{align*}
Then, $V$ is a convex neighborhood of $0 \in \mathcal{F}$. Moreover, by varying $\epsilon, k$, and the $x_i$'s in $\mathcal{K}$, we obtain a convex local base of $0 \in \mathcal{F}$. See Proposition 3.12. of \cite{Brezis2011}. On the other hand, the Banach–Alaoglu–Bourbaki Theorem states that the sets $S^{-}\left(\|\cdot\|_{\mathcal{F}}, \alpha\right)$ are $\tau$-compact for all $\alpha > 0$. In particular, these sets are also $\tau$-closed, since they are convex and closed in the strong topology associated with $\|\cdot\|_{\mathcal{F}},$ see Theorem 3.7 of \cite{Brezis2011}, which implies that their complements, namely
\begin{align*}
S^{-}\left(\|\cdot\|_{\mathcal{F}}, \alpha\right)^{c} = \left\{\mu \in \mathcal{F}: \|\mu\|_{\mathcal{F}} > \alpha\right\},
\end{align*}
are open in $\tau$. This latter fact means that the function $\|\cdot\|_{\mathcal{F}}$ is lower semi-continuous with respect to the topology $\tau$.

Regarding $\mathit{E}$, we will assume that it is a finite-dimensional Hilbert space of dimension $m$, so that $\mathit{E} \times \mathit{E} \times \cdots \times \mathit{E}$ ($N$ times), where $N$ is the number of samples in the training set, is a Hilbert space of dimension $mN$. In $\mathit{E} \times \mathit{E} \times \cdots \times \mathit{E}$, we consider the inner product $\left\langle\cdot, \cdot \right\rangle_{\mathit{E} \times \mathit{E} \times \cdots \times \mathit{E}} = \left\langle\cdot, \cdot \right\rangle_{\mathit{E}} + \cdots + \left\langle\cdot, \cdot \right\rangle_{\mathit{E}}$ ($N$ times).

Now, we define the operator $\mathcal{A}: \mathcal{F} \rightarrow \mathit{E} \times \mathit{E} \times \cdots \times \mathit{E}$ given by \[\mu \mapsto \left(\phi(x_1)(\mu), \cdots, \phi(x_N)(\mu)\right),\]
which is a linear and continuous operator such that $\text{range} \mathcal{A}$ is closed, and therefore, it is a Hilbert subspace of $\mathit{E} \times \mathit{E} \times \cdots \times \mathit{E}$ of dimension $\leq mN$. We denote by $\mathcal{H} = \text{range} \mathcal{A}$. 

Finally, we assume that the loss function $\mathcal{L}: \mathit{E}\times \mathit{E} \rightarrow \mathbb{R}$ is convex, proper, and coercive with respect to the second variable. In this sense, if we define $F: \mathcal{H} \rightarrow \mathbb{R}$ by
\[
F(w) = \frac{1}{N}\sum^{N}_{i = 1}\mathcal{L}(y_i, w_i)
\]
for all $w = (w_1,\dots,w_N) \in \mathcal{H}$, then this function is convex, proper, and coercive, and hence lower semi-continuous. Moreover,
\[
F(\mathcal{A}\mu) = \frac{1}{N}\sum^{N}_{i = 1}\mathcal{L}(y_i, \phi(x_i)(\mu)), \quad \forall \mu \in \mathcal{F}.
\]

Once we have imposed the conditions explained in the preceding paragraphs, we are in a position to state a result regarding the existence of sparse solutions to the problem \eqref{atpmu}, which we will hereafter refer to as the abstract training problem associated with $\mathcal{F}$. We then have the following theorem.
\begin{theorem}\label{TeoSpaSol}
Let $\mathcal{B}(\mathcal{X}, \mathit{E})$ be a GRKBS with configuration $\left(\mathcal{F}, \phi \right)$ satisfying the following conditions:
\begin{description}
\item[H0] There exists a topology $\tau$ in $\mathcal{F}$ that makes $(\mathcal{F}, \tau)$ a locally convex topological space and the norm $\|\cdot\|_{\mathcal{F}}$ coercive,
\item[H1] $\mathit{E}$ is a finite-dimensional Hilbert space of dimension $m,$
\item[H2] $\mathcal{L}: \mathit{E}\times \mathit{E} \rightarrow \mathbb{R}$ is convex, proper, and coercive with respect to the second variable,
\end{description}
then the abstract training problem associated with $\mathcal{F}$ \eqref{atpB} has a sparse solution in the form:
\begin{align}\label{spasol}
\mu^{*} = \sum^p_{i = 1}\gamma_i \mu_i,
\end{align}
where $p \leq mN,$ $\gamma_i > 0$ such that $\|\mu^{*}\|_{\mathcal{F}} = \sum^p_{i = 1}\gamma_i$, and $\mu_i \in \text{Ext}\left(B_{\mathcal{F}}\right)$, where $B_{\mathcal{F}}$ is the closed unit ball in $\mathcal{F}$, and $\text{Ext}\left(B_{\mathcal{F}}\right)$ are the extremal points of this unit ball. 
\end{theorem}
\begin{proof}
The proof of this theorem is based on the fact established from the previous comments and the assumed conditions \textbf{H0}, \textbf{H1}, and \textbf{H2}, that the problem \eqref{atpmu} can be formulated as a more general variational problem, and that in Theorem 3.3 of \cite{Bredies2020} it is shown to have a sparse solution in the form \eqref{spasol}.
\end{proof}

Based on Theorem \ref{TeoSpaSol}, we provide the following definition.
\begin{definition}
Let $\left\{\mu_1, \mu_2,\ldots, \mu_p\right\} \subset \text{Ext}\!\left(B_{\mathcal{F}}\right)$ be a finite subset of $p \leq nM$ extremal points of the unit ball, and let $\left\{\gamma_1, \gamma_2,\ldots, \gamma_p\right\} \subset \mathbb{R}^{+}$ be given. We call the family of functions of the form
\begin{align*}
f_{\mu^{*}} = \sum^{p}_{i = 1}\gamma_i f_{\mu_i}
\end{align*}
a \textbf{one-layer sparsity neural structure}. The functions $f_{\mu_i}$ are referred to as \textbf{structural neural units}.
\end{definition}

\subsection{Physically Modeled Neural Networks}
In this section, we present an example of a GRKBS. In this context, when an AbsNN incorporates a GRKBS of this type as part of the sequence of GRKBSs that constitute the abstract network, we refer to it as a \emph{Physically Modeled Abstract Neural Network}. 

We consider the following boundary value problem as a model for brain activity that incorporates the distribution of electric charges. Let $\Omega \subset \mathbb{R}^n$ be a region with smooth boundary $\partial \Omega$. 
\begin{align}\label{eqn8}
\left\{
\begin{array}{ll}
\mathcal{L}u = f & \text{in } \Omega, \\\\
\frac{\partial u}{\partial \mathbf{n}} = 0 & \text{on } \partial \Omega,
\end{array}
\right.
\end{align}
where the operator is defined by
\[
\mathcal{L} u = \mathrm{div}\!\left(k(y)\nabla u\right) - a(y)u,
\]
with $k \in C^1(\overline{\Omega})$, $a \in C(\overline{\Omega})$, and such that $k(y) \geq k_0 > 0$ and $a(y) > 0$ for all $y \in \Omega$.

A standard result states that for every $f \in L^2(\Omega)$, the boundary value problem \eqref{eqn8} admits a unique weak solution $u \in H^1(\Omega)$ satisfying
\begin{align*}
\|u\|_{H^1(\Omega)} \leq C \|f\|_{L^2(\Omega)},
\end{align*}
where $C$ is a constant independent of $f$.

Now, one can define a linear and bounded operator $\mathcal{K}: L^2(\Omega) \rightarrow H^1(\Omega)$ which assigns to each $f$ the unique weak solution of \eqref{eqn8}. Moreover, this operator is invertible as a map from $L^2(\Omega)$ onto $H^1(\Omega)$. When restricted to $H^1(\Omega)$, $\mathcal{K}$ is self-adjoint, positive, and compact.

In addition, the differential operator $\mathcal{L}$ admits a complete orthonormal system of eigenfunctions $\left\{\psi_i\right\}_{i=1}^{\infty}$ in $L^2(\Omega)$. For any function $f \in L^2(\Omega)$, we have the expansion
\begin{align*}
f = \sum_{i=1}^{\infty} \left(f, \psi_i\right)_{L^2(\Omega)} \psi_i,
\end{align*}
satisfying the identity
\begin{align*}
\|f\|^2_{L^2(\Omega)} = \sum_{i=1}^{\infty} \left(f, \psi_i\right)^2_{L^2(\Omega)}.
\end{align*}

Furthermore, one can define the projection operator $\mathcal{P}_m : H^1(\Omega) \rightarrow H^1(\Omega)$, which maps each $u \in H^1(\Omega)$ to
\[
\sum_{i=1}^{m} \left(u, \psi_i\right)_{L^2(\Omega)} \psi_i,
\]
and it can be shown that this operator is bounded with a bound independent of $m$. We denote by $E_m = \text{span}\left(\psi_1,\ldots,\psi_m\right) \subset H^1(\Omega)$.

Our next step is to define a GRKBS that we refer to as the \textbf{\textit{Sensory Encoding Projector}}, whose role is to map encoded signals received from the environment via multiple sensory modalities into a distribution of electric charges in the brain.

To this end, we define $\mathcal{B}_{\mathrm{Sensor}}$ as a GRKBS of vector-valued functions $f: \mathcal{X} \rightarrow L_2(\Omega)$, where $\mathcal{F}$ is the characteristic space and $\phi: \mathcal{X} \rightarrow L_{\text{op}}\!\left(\mathcal{F}, L_2(\Omega)\right)$ is the characteristic map, such that
\begin{align*}
\mathcal{B}_{\mathrm{Sensor}} &= \left\{f_{\mu} : \mu \in \mathcal{F} \right\}, \quad \text{with } f_{\mu}(x) = \phi(x)(\mu) \in L_2(\Omega), \\
\|f\|_{\mathcal{B}_{\mathrm{Sensor}}} &= \inf \left\{ \|\mu\|_{\mathcal{F}} : f = f_{\mu} \right\}, \quad f \in \mathcal{B}_{\mathrm{Sensor}}.
\end{align*}
The structure associated with $\mathcal{B}_{\mathrm{Sensor}}$ is described by the quadruple $\left(\mathcal{X}, L_2(\Omega), \mathcal{F}, \phi\right)$.

Now, starting from $\mathcal{B}_{\mathrm{Sensor}}$, we construct another GRKBS, denoted by $\mathcal{B}_{\mathrm{PMANN}}$ and referred to as the \textbf{\textit{Physically Modeled Abstract Neural Network}}. This is a GRKBS of functions $f: \mathcal{X} \rightarrow E_m$, with characteristic space $\mathcal{F}$ and characteristic map
\begin{align*}
&\widetilde{\phi} : \mathcal{X} \rightarrow L_{\text{op}}\!\left(\mathcal{F}, E_m\right), \\
&\widetilde{\phi}(x)(\mu) = \mathcal{P}_m \circ \mathcal{K}(\phi(x)(\mu)).
\end{align*}
It is clear that for each $x \in \mathcal{X}$, $\widetilde{\phi}(x)$ is a linear operator from $\mathcal{F}$ into $E_m$. It remains to show that $\widetilde{\phi}(x)$ is bounded. In fact, we have
\begin{align*}
\|\mathcal{P}_m \circ \mathcal{K}(\phi(x)(\mu))\|_{H^1(\Omega)} 
&\leq M \|\mathcal{K}(\phi(x)(\mu))\|_{H^1(\Omega)} \\
&\leq MC \|\phi(x)(\mu)\|_{L^2(\Omega)} \\
&\leq MC \|\phi(x)\|_{L_{\text{op}}(\mathcal{F}, L^2(\Omega))} \|\mu\|_{\mathcal{F}}.
\end{align*}
Thus, $\mathcal{B}_{\mathrm{PMANN}}$ defines a GRKBS with structure $\left(\mathcal{X}, E_m, \mathcal{F}, \widetilde{\phi}\right)$, where $E_m$ is a finite-dimensional Hilbert space of dimension $m$.

We have developed a GRKBS designed, to some extent, to replicate the way in which the brain assimilates, encodes, and responds to environmental stimuli. The input of information triggers physio-electromagnetic processes that activate neurons, which then propagate through synapses and ionic exchange between the interior and exterior of the neurons, thereby activating a wide region of the brain.

By applying Theorem~\ref{TeoSpaSol} to $\mathcal{B}_{\mathrm{PMANN}}$, one obtains, as a sparse solution of the ATP, a sparse neural structure that is already computable. Therefore, this structure can be regarded as a layer that may be integrated with other types of layers to build deep architectures intended for the training of artificial intelligence models.

Finally, we state the following result. 
\begin{theorem}
Let $\{(x_i, y_i)\}_{i = 1}^N \subset \mathcal{X}\times E_m$ be a training dataset, and let $L: E_m\times E_m \rightarrow \mathbb{R}$ be a loss function satisfying the conditions of Theorem~\ref{TeoSpaSol}. Then the abstract training problem 
\begin{align*}
\inf_{f \in \mathcal{B}_{\mathrm{PMANN}}}\left(\frac{1}{N}\sum^{N}_{i = 1} L(y_i, f(x_i)) + \|f\|_{\mathcal{B}_{\mathrm{PMANN}}}\right) 
\end{align*}
admits a solution $f^{*}$ of the form
\begin{align*}
f^{*}(x) = \sum^p_{i = 1}\gamma_i \,\mathcal{P}_{m}\circ \mathcal{K}(\phi(x)(\mu_i)),
\end{align*}
where $p \leq mN$, $\gamma_i > 0$ with $\|f^{*}\|_{\mathcal{F}} = \sum^p_{i = 1}\gamma_i$, and $\mu_i \in \text{Ext}\!\left(B_{\mathcal{F}}\right)$, where $B_{\mathcal{F}}$ is the closed unit ball in $\mathcal{F}$ and $\text{Ext}\!\left(B_{\mathcal{F}}\right)$ denotes its set of extreme points.
\end{theorem}

\subsection*{Declaration of Generative AI and AI-assisted technologies in the writing process}

During the preparation of this work the authors used \textit{ChatGPT} in order to improve the readability and language of the manuscript. After using this tool, the authors reviewed and edited the content as needed and take full responsibility for the final version of the publication.

\bibliographystyle{unsrtnat}

\bibliography{bib}

\end{document}